\documentclass[reqno,a4paper]{amsart}
\usepackage[colorlinks,citecolor=blue,urlcolor=blue]{hyperref}
\usepackage{yfonts}
\usepackage{enumerate}

\usepackage{tikz}
\usetikzlibrary{intersections,shapes,trees,arrows,spy,positioning,decorations,arrows.meta,decorations.pathreplacing,patterns,calc,decorations.markings}
\tikzset
{marking1/.style=
	{decoration=
		{markings,
			mark= between positions 0.03 and 0.97 step 5 mm with {\arrow[line width=0.5pt]{>}}
		},
		postaction=decorate
	}
}
\usepackage{mathrsfs}

\makeatletter
\def\timenow{\@tempcnta\time
	\@tempcntb\@tempcnta
	\divide\@tempcntb60
	\ifnum10>\@tempcntb0\fi\number\@tempcntb
	\multiply\@tempcntb60
	\advance\@tempcnta-\@tempcntb
	:\ifnum10>\@tempcnta0\fi\number\@tempcnta}
\makeatother

\usepackage[ddmmyyyy,hhmmss]{datetime}
\usepackage{mathtools}
\usepackage{euscript}
\usepackage[numbers,square,sort]{natbib}
\usepackage{amssymb,amsmath,amsthm,dsfont}
\usepackage{graphics,graphicx}
\usepackage[T1]{fontenc}
\usepackage{color}
\usepackage{epsfig,psfrag}
\numberwithin{equation}{section}

\newtheorem{theo}{Theorem}[section]
\newtheorem{prop}[theo]{Proposition}

\newtheorem{cond}[theo]{Condition}
\newtheorem{lemme}[theo]{Lemma}

\newtheorem{remark}[theo]{Remark}

\newtheorem{defi}[theo]{Definition}

\DeclareFontFamily{U}{mathb}{}
\DeclareFontShape{U}{mathb}{m}{n}{<-5.5> mathb5 <5.5-6.5> mathb6 
<6.5-7.5> mathb7 <7.5-8.5> mathb8 <8.5-9.5> mathb9 <9.5-11> mathb10 
<11-> mathb12}{}
\DeclareSymbolFont{mathb}{U}{mathb}{m}{n}
\DeclareFontSubstitution{U}{mathb}{m}{n}
\DeclareMathSymbol{\blackdiamond}{\mathbin}{mathb}{"0C}
\newcommand{\sel}{\vec{\mathfrak{s}}}
\def\E{\mathbb{E}}
\def\N{\mathbb{N}}
\def\P{\mathbb{P}}

\def\R{\mathbb{R}}

\newcommand\Ac{{\mathscr A}}
\newcommand\Bc{{\mathscr B}}

\newcommand\1{{\bf 1}}
\def\dd{\textnormal{d}}
\newcommand{\bindist}[2]{\textrm{Bin}({#1},{#2})}
\newcommand{\berno}[1]{\textrm{Bern}({#1})}
\newcommand{\hypdist}[3]{\textrm{Hyp}({#1},{#2},{#3})} 
\newcommand{\Eb}  {{\mathbb E}}

\newcommand{\Nb}  {{\mathbb N}}

\newcommand{\Rb}  {{\mathbb R}}

\newcommand{\Ds} {{\mathcal D}}

\newcommand{\Ms} {{\mathcal M}}

\newcommand*{\affmark}[1][*]{\textsuperscript{#1}}

\title[Bernstein duality revisited]{Bernstein duality revisited: frequency-dependent selection, coordinated mutation and opposing environments}

\author{F. Cordero\affmark[1,2] \and S. Hummel\affmark[3] \and G. V{\'e}chambre\affmark[4]}
\address{\newline \affmark[1]BOKU University, Institute of Mathematics, Vienna, Austria
\newline\affmark[2]Faculty of Technology, Bielefeld University, Bielefeld, Germany\newline
\affmark[3]Department of Health Science and Technology, ETH Z\"urich, Z\"urich, Switzerland\newline 
\affmark[4]Academy of Mathematics and Systems Science, Chinese Academy of Sciences, No. 55, Zhongguancun East Road, Haidian District, Beijing, China}
\email{\affmark[1]fernando.cordero@boku.ac.at, affmark[2]fcordero@techfak.uni-bielefeld.de, \affmark[3]shummel@hest.ethz.ch,\newline \affmark[4]vechambre@amss.ac.cn}

\date{\today}

\begin{document} 
\maketitle

\begin{abstract}
	This paper investigates the long-term behavior of a class of $\Lambda$-Wright--Fisher processes incorporating frequency-dependent selection, coordinated (bidirectional) selection, as well as individual and coordinated mutation.
	Our primary analytical tool is Bernstein duality, a generalization of moment duality. 
	We introduce the corresponding dual process and establish the relevant duality relation.
	Without mutation, this work complements earlier studies that employed moment duality, Siegmund duality or other methods to classify the long-term behavior of similar processes.
	Notably, the current analysis encompasses parameter regimes that model bidirectional selection, a scenario that has proven challenging to analyze using moment duality. 
	In the presence of mutation, we establish the ergodic properties of the process.
	
\end{abstract}

\bigskip

\section{Introduction}
Populations influenced by evolutionary forces such as natural selection, mutation, and changing environments can exhibit unique long-term patterns in the frequency of different types of individuals. 
Even in simple scenarios where only two distinct types are present, these evolutionary dynamics can lead to a variety of long-term behaviors, including stability, oscillation, or the dominance of one type over the other.
This has recently been demonstrated in \cite{cordhumvech2022} in the setting of two-type $\Lambda$-Wright--Fisher models subject to frequency-dependent selection and coordinated bidirectional selection. 
One of the main results in \cite{cordhumvech2022} is a characterization of possible long-term behaviors under the assumption that the strength of neutral reproductions and of the coordinating forces are not too strong (see \cite[Thm 2.1]{cordhumvech2022} for the precise statement). 
As a consequence of this result, a non-trivial class of models exhibiting coexistence was uncovered. 
Furthermore, it was shown that in such models, coexistence requires frequency-dependent selection. 

In contrast to~\cite{cordhumvech2022}, previous studies considered parameter regimes in which only fixation and extinction were observed. For example, \cite{ParBah15}, \cite{foucart2013impact} and \cite{Gri14} considered $\Lambda$-Wright--Fisher models subject to genic selection. 
The results therein were extended by~\cite{GS18} to models incorporating a special form of frequency-dependent selection (fittest type wins) and by~\cite{casa19e} to populations evolving in an environment favorable to the fittest type (see also \cite{CorderoVechambre}). 
A common theme in these papers was \emph{moment duality}, a simple relation connecting the moments of the type-frequency process to an ancestral process, the moment dual, which counts the number of potential ancestors. 
The existence of a moment dual in this setting is closely related to the property that all modeled forces always favor the same type. 
In particular, bidirectional selection models---where the selection of a type can be either positive or negative depending on the current state of the population or the environment---were not considered.
A first attempt to study models breaking this rule was based on a genealogical approach in~\cite{Cordero2022}, 
where the long-term behavior of $\Lambda$-Wright--Fisher process under general polynomial frequency-dependent selection was analyzed. 
In this context, the strength of selective pressure and the type that is in selective advantage at a given time depend on the type composition of the population. 
The main contribution in \cite{Cordero2022} was to generalize the concept of moment duality to a more general relation, the \emph{Bernstein duality} (see also \cite{Koske24} for more general frequency-dependent selection forms). 
The latter relation still enables computing the moments of the type-frequency process in terms of a tractable ancestral structure. 
Provided the strength of neutral reproductions is sufficiently large in comparison to the strength of selection, it was shown that the probabilities of fixation and extinction are both positive and complementary (no coexistence). 
Another contribution outside the moment duality setting was given in~\cite{Vec2023}, where a Wright--Fisher process with coordinated bidirectional selection was analyzed using combinatorial properties of the underlying ancestral structure. 
The analysis in \cite{cordhumvech2022} uses the Siegmund dual, whose existence is guaranteed under mild conditions (see \cite{Siegmund1976}). 
This comes at the cost that an interpretation of the process in terms of an ancestral structure is not obvious.
Moreover, the analysis in \cite{cordhumvech2022} could only be conducted under specific integrability conditions on the model-parametrizing measures, which excluded several commonly studied special cases like the Wright-Fisher diffusion.

When bidirectional mutation is introduced into the model, permanent extinction and fixation is no longer possible.
The focus then typically shifts to existence and characterization of a stationary distribution, and to establishing convergence rates. Once again, moment duality turns out to be a powerful tool to answer these questions, see e.g. \cite{CorderoVechambre}, \cite{BEH21}, and \cite{GKT}. 
In the present work we consider individual and coordinated forms of mutation (coordination as in \cite{GKT}). 
In addition, we allow for coordinated bidirectional selection and for frequency-dependent selection to be of general polynomial form. 
This setting precludes, in general, a moment duality.

The model components capturing neutral reproduction, frequency-dependent selection and individual mutation are well-established.
Selection coordinated by an environment can be used to model strong, instantaneous selection events.
For example, it has been observed that lizards with long fingers experience significant positive selection during hurricanes that impact their habitat~\cite{donihue_hurricane_induced_2018};
and selection coordinated by an environment (in this case, the occurrence of a hurricane) can be used to model the frequency of short- and long-fingered lizards, as noted in~\cite[Rem. 3.3]{casa19e}. 
In our context, long-fingered lizards can also be modeled as experiencing negative selection under normal conditions (a particular instance of bidirectional selection).
A biological motivation for coordinated mutation is the modeling of systems with two phenotypes that exhibit adaptive switching in response to environmental changes~\cite{mitchell_adaptive_2009}.
For example, in brewer's yeast, genes have been observed to be activated during heat shocks to protect against oxidative stresses, which often follows in the wine production process~\cite{bleuven_molecular_2016}. 
This type of switching, here coordinated by heat shocks, can be modeled as coordinated mutation.

The purpose of this paper is threefold. 
First, to extend the notion of Bernstein duality from \cite{Cordero2022} to a broader class of $\Lambda$-Wright--Fisher processes that incorporate frequency-dependent and coordinated (bidirectional) selection, as well as both individual and coordinated mutation. To achieve this, we employ a combinatorial approach to establish the duality, 
which we believe simplifies the proof and makes it more intuitive than a direct computation.

Second, we utilize Bernstein duality to characterize the limiting behavior of these processes in the absence of mutation, under the assumption that neutral reproduction is sufficiently strong. This closes a gap in the boundary classification in \cite{cordhumvech2022}.

Third, we apply Bernstein duality in the presence of mutation to characterize the long-term behavior of the forward process. In particular, we focus on the ergodic properties of the model. The corresponding results are largely novel in this generality.

In \cite{Cordero2022}, \cite{Koske24} and in this manuscript, Bernstein duality is used to analyze population models.
In this work, we also aim to clarify the original concept to facilitate its use.
For instance, we have adapted the proofs from \cite{Cordero2022} to our setting and have simplified proofs for some results where possible. 
Additionally, to ensure a self-contained presentation, we have included a dedicated section on the motivation and biological interpretation of the processes involved.

The paper is organized as follows. In Section \ref{S2} we introduce the class of $\Lambda$-Wright--Fisher processes we consider in the present paper.
Our main results are stated in Section~\ref{sec:statementmainresults}, in particular the Bernstein duality relating the $\Lambda$-Wright--Fisher process and the so-called Bernstein coefficient process, and the results on the long term behavior of the $\Lambda$-Wright--Fisher process. 
Section~\ref{sec:discussionprevresults} discusses previous results and explains the novelty of the current work. In Section \ref{sec:Bernstein} we introduce the finite population (individual-based) counterpart of our model to provide intuition on the dynamics. This provides insight on the ancestral structure of the process and allows to formally derive the Bernstein coefficient process. In Section~\ref{sec:lcp} we establish some properties of the Bernstein coefficient process and prove the Bernstein duality. In Section~\ref{sec:proofmainres} we prove the remaining  main results.
\subsection{Notation}\label{sec:prequel}
The positive and non-negative integers are denoted by $\N$ and $\N_0\coloneqq\N\cup\{0\}$, respectively. The non-negative real numbers are denoted by~$\R_+$. For $n\in \N$, we define $[n]\coloneqq\{1,\ldots,n\}$, $[n]_0\coloneqq[n]\cup\{0\}$ and $(n]\coloneqq[n]\setminus\{1\}.$
For any Borel set $S\subset\Rb$, $\Ms_f(S)$ is the set of finite measures on $S$.\\
For $n\in \N_0$ and $x\in [0,1]$, we write $B\sim \bindist{n}{x}$ if $B$ is a binomial random variable with parameters $n$ and $x$, i.e. $\P(B=i)= { \tiny \binom{n}{i}} x^i (1-x)^{n-i}$ for $i\in [n]_0$. We write $B\sim \berno{x}$ if $B$ is a Bernoulli random variable with parameters $x$, i.e. $B\sim \bindist{1}{x}$. 
For $n,k,j\in \N_0$ with $n\geq k\vee j$, we write $K\sim \hypdist{n}{k}{j}$ if~$K$ is a hypergeometric random variable with parameters $n,k$, and $j$, i.e. $\P(K=i)= {\tiny \binom{k}{i} \binom{n-k}{j-i}/\binom{n}{j}}$, $i\in \{ 0\vee (j+k-n),\ldots, k\wedge j\}.$
For $n\in \N$, $\ell\in[n]$ and $i\in[n+\ell]$, we define the distribution $\mathrm{HP}(n+\ell,\ell,i)$ via the following \emph{hypergeometric pairing}. Consider an urn containing $n+\ell$ balls with $i$ of them being red and $n+\ell-i$ being blue. If $\ell$ pairs of balls are formed uniformly at random; the other $n-\ell$ balls are kept unpaired, then $R\sim \mathrm{HP}(n+\ell,\ell,i)$ denotes the number of groups (pairs and singletons) containing at least one red ball. 
%%%%%%%%%%%%%%%%%%%%%%%%%%%%%%%%%%%%%%%%%%%%%%%%%%%%%%%%%%%%%%%%%%%%%
%%%%%%%%%%%%%%%%%%%%%%%%%%%%%%%%%%%%%%%%%%%%%%%%%%%%%%%%%%%%%%%%%%%%%

\section{Model and main results}\label{S2}
The main focus of this study is to investigate the long-term behavior of an infinite, two-type population undergoing random reproduction, frequency-dependent selection, mutation (individual and coordinated) and coordinated selective reproduction (triggered by random fluctuations of an environment).
Let $a$ and $A$ denote the two types of individuals in the population, corresponding, for example, to specific alleles.
We write $X_t$ for the proportion of type-$a$ individuals present at time $t$, and we assume $X\coloneqq (X_t)_{t\geq 0}$ evolves according to the stochastic differential equation (SDE)
\begin{align}
	\dd  X_t \  = &\  \sqrt{\Lambda(\{0\})X_t(1-X_t)}\, \dd W_t+  \int_{(0,1]\times [0,1]}z\big(\mathds{1}_{\{u\leq X_{t-}\}}(1-X_{t-})-\mathds{1}_{\{u>X_{t-}\}} X_{t-} \big)\tilde{N}(\dd t, \dd z,\dd u)\nonumber\\
	&+\left( d_{{\sel}}(X_{t})+\theta_a (1-X_t)-\theta_A X_t\right)\dd t + \int_{(-1,1)} \lvert z\rvert (\mathds{1}_{\{z\geq 0\}}(1-X_{t-})-\mathds{1}_{\{z< 0\}}X_{t-}) M(\dd t, \dd z) \nonumber\\
	&+ \int_{(-1,1)} zX_{t-}(1-X_{t-})S(\dd t, \dd z). \label{eq:SDEWFP}
\end{align}
where the building blocks, here grouped according to their biological interpretation, are:

\noindent (a) \emph{neutral reproduction} driven by a standard Brownian motion $W\coloneqq (W_t)_{t\geq 0}$ and a compensated Poisson measure $\tilde{N}(\dd t,\dd z,\dd u)$ on $[0,\infty)\times(0,1]\times[0,1]$ with intensity $\dd t\times z^{-2}\Lambda(\dd z)\times \dd u$, where $\Lambda\in\Ms_f([0,1])$,	 

\noindent (b) \emph{frequency-dependent selection} contributing to the the drift term via the polynomial $d_{\vec{s}}$ given by
\begin{equation*}
	d_{\sel}(x)\coloneqq\sum_{\ell=2}^{\kappa}\beta_\ell\sum_{i=0}^\ell \binom{\ell}{i}x^i(1-x)^{\ell-i}\Big(p_{i}^{(\ell)}-\frac{i}{\ell}\Big),\quad x\in[0,1],
\end{equation*}
where $\kappa\in\Nb$ with $\kappa>1$, $\beta=(\beta_\ell)_{\ell\in(\kappa]}\in \R_{+}^{\kappa-1}$, and $p\coloneqq (p_i^{(\ell)}:\ell\in (\kappa], i\in[\ell]_0)$ with $p^{(\ell)}_i\in[0,1]$ and $p_0^{(\ell)}=0=1-p_\ell^{(\ell)}$; these parameters are summarized in the vector $\sel\coloneqq (\kappa,\beta,p)$. 
Any polynomial of the form $x(1-x)\sigma(x)$, with a polynomial $\sigma$, can be expressed in this way~\cite{Cordero2022}. In particular, this covers several classical forms of selection, see Section \ref{sec:discussionprevresults}. 

\noindent (c) \emph{coordinated selective reproductions} driven by a Poisson measure $S(\dd t, \dd z)$ on $[0,\infty)\times(-1,1)$ with intensity $\dd t \times \lvert z\rvert^{-1}\mu(\dd z)$, where $\mu\in\Ms_f([-1,1])$ with $\mu(\{0\})=0$. In line with previous literature, we refer to~$S$ also as the environment for coordinated selection.

\noindent (d) \emph{individual mutation} rates $\vec{\theta}\coloneqq(\theta_a,\theta_A)\in \Rb_+^2$, contributing to the drift term, 

\noindent (e) \emph{coordinated mutation} driven by a Poisson measure $M(\dd t, \dd z)$ on $[0,\infty)\times(-1,1)$ with intensity $\dd t \times \lvert z\rvert^{-1} \nu(\dd z)$, where $\nu\in\Ms_f([-1,1])$ with $\nu(\{0\})=0$.

We assume that $W$, $\tilde{N}$, $M$ and $S$ are independent.
The underlying processes at the individual-based level are explained in Section~\ref{sec:Bernstein}. 
The existence and pathwise uniqueness of strong solutions of SDE~\eqref{eq:SDEWFP} can be shown analogously to \cite[Lemma 3.2]{Cordero2022} (see also \cite[Prop. A3]{cordhumvech2022}). We refer to $X$ as the $\Lambda$-Wright--Fisher process with frequency-dependent selection $\sel$, individual and coordinated mutation {$(\vec{\theta},\nu)$} and coordinated selection $\mu$. We denote by $\P_x$ the probability measure associated with $X$ under the initial condition $X_0=x\in[0,1]$, and by $\E_x$ the associated expectation. 

In this work, we provide a unifying framework for the study of the long-term behavior for a general class of $\Lambda$-Wright--Fisher processes using Bernstein duality. In particular, we advocate for using Bernstein duality in cases where Siegmund duality has not been fruitful to analyze the long-term behavior of~\eqref{eq:SDEWFP}. 
To this end, we complement the results of~\cite{cordhumvech2022} derived via Siegmund duality by considering a parameter regime that covers the cases when the assumptions in \cite{cordhumvech2022} are not satisfied. 
Some of these cases have been already treated in~\cite{foucart2013impact,GS18, CorderoVechambre, casa19e} using an approach based on the existence of a moment dual, which excludes cases where
\begin{equation}
	\left(\mu(0,1)\vee\max_{x\in[0,1]}d_{\vec s}(x)\right)\left(\mu(-1,0)\vee\max_{x\in[0,1]}-d_{\vec s}(x)\right)>0\qquad \textit{(bidirectional selection)}. \label{eq:bisel}
\end{equation}
Thus, our results fill the gap in the literature for processes that were not amenable to study by moment duality and that do not satisfy the assumptions of \cite{cordhumvech2022}.

More concretely, our main results are establishing: 
i) Bernstein duality between~$X$ and a functional of a branching-coalescing structure (reminiscent of an ancestral selection graph), which is new for processes with coordinated selection and mutation, and ii) the asymptotic properties of $X$ under the assumption
\begin{equation}\label{Assump}
	\sum_{\ell=2}^\kappa \beta_\ell (\ell-1)+\mu(-1,1)<\int_{[0,1]}|\log(1-z)|\frac{\Lambda(\dd z)}{z^2}+\nu(-1,1)+\theta_a+\theta_A.
\end{equation}

This includes ii-a) a representation for the absorption probability of~$X$ in~$0$ and~$1$ if 
\begin{equation}
	\vec{\theta}=0 \textit{ and } \nu=0 \qquad \textit{(no mutation)}\label{eq:nomut}
\end{equation}
and ii-b) the ergodicity of~$X$ and a representation for the moments of its stationary distribution if
\begin{equation}
	\theta_a+\nu(0,1)>0 \textit{ and } \theta_A+\nu(-1,0)>0\qquad \textit{(bidirectional mutation)}. \label{eq:bimut}
\end{equation}
We discuss how these results compare to the existing literature in Section \ref{sec:discussionprevresults}.

\subsection{Statement of main results}\label{sec:statementmainresults}
To state the results, we first need to introduce the dual process. Its transitions are expressed in terms of finite-dimensional linear operators $C^{n,k}$, $D^{n,l}$, $M_c^{n,k}$, and $S_c^{n,\ell}$, $n\in \N$, $k\in(n]$, $l\in(\kappa]$, $\ell\in[n]$, $c\in\{a,A\}$. While the technical definition of these operators is important, we believe that discussing it in detail at this stage of the manuscript may not be particularly enlightening, and hence we postpone it to Definition~\ref{def:bramutoperators}. Starting from any vector $v\in \cup_{n\in\N_0}\R^{n+1} \eqqcolon\R^\infty$, the states visited by the dual process will be obtained by successive (compatible) compositions of the above mentioned operators and $v$ (see Definition \ref{def:BCP} below). It is convenient for us to think of the dual process as a Markov process on a discrete state space. For this, we introduce some countable sets that will be invariant to its dynamics.
%To state the results, we first need to introduce the dual process, which is characterized by transitions expressed through linear operators that are connected to a branching-coalescing structure.
%While the technical definition of these operators is important, we believe that discussing it in detail at this stage of the manuscript may not be particularly enlightening. 
%For now, we note that for $n\in \N$, $k\in(n]$, $l\in(\kappa]$, $\ell\in[n]$, $c\in\{a,A\}$, there are well-defined finite-dimensional linear operators $C^{n,k}$, $D^{n,l}$, $M_c^{n,k}$, and $S_c^{n,\ell}$ that are precisely stated in Definition~\ref{def:bramutoperators}. 
%For any starting value $v\in \cup_{n\in\N_0}\R^{n+1} \eqqcolon\R^\infty$, the dual process at a given time will be obtained by a finite compatible composition of the above mentioned operators and $v$ (see Definition \ref{def:BCP} below). It is convenient to us to see the dual process as a Markov process on a discrete state space. For this, we define some countable sets that are invariant. 
For $n \in \N$ and $v\in \R^{n+1}$, we let $\mathbb{M}_{V}$ (resp. $\mathbb{M}^0_{V}$) be the set of operators from $\R^{n+1}$ to $\R^k$, for some $k \in \N$, that can be obtained as the composition of a finite number of (compatible) operators of the form $C^{\cdot,\cdot}$, $D^{\cdot,\cdot}$, $M_{\cdot}^{\cdot,\cdot}$, and $S_{\cdot}^{\cdot,\cdot}$ (resp. $C^{\cdot,\cdot}$, $D^{\cdot,\cdot}$, and $S_{\cdot}^{\cdot,\cdot}$), including the identity operator of $\R^{n+1}$ corresponding to an empty composition. Let $C_V(v), C^0_V(v) \subset \R^\infty$ be defined by 
\begin{equation}
	C_V(v)\coloneqq\left\{M v\in \R^{\infty}: \ M\in \mathbb{M}_{V}\right\}, \qquad C^0_V(v)\coloneqq\left\{M v\in \R^{\infty}: \ M\in \mathbb{M}^0_{V}\right\}. \label{defcvv}
\end{equation}
The set $C_V(v)$ (resp. $C_V^0(v)$) is by definition invariant by composition of the operators $C^{\cdot,\cdot}$, $D^{\cdot,\cdot}$, $M_{\cdot}^{\cdot,\cdot}$, and $S_{\cdot}^{\cdot,\cdot}$ (resp. $C^{\cdot,\cdot}$, $D^{\cdot,\cdot}$, and $S_{\cdot}^{\cdot,\cdot}$). Since they are countable we equip them with the discrete topology and consider convergence in distribution relative to that topology. 
%The set $C^0_V(v)$ is invariant in the case without mutation. 

\begin{defi}[Bernstein coefficient process]\label{def:BCP}
	
	Let $\xi\in\Rb^\infty$ and $v\in C_V(\xi)$. 
	The {$(\Lambda,\mu,\vec{\theta},\nu,\sel)$-Bernstein coefficient process} starting at $v$ is the $C_V(\xi)$-valued continuous-time Markov chain $V\coloneqq(V_r)_{r\geq 0}$ such that $V_0=v$ and that transitions occur from $w\in C_V(\xi) \cap \R^{n+1}$ to:
	\begin{itemize}
		\item $C^{n,k}w\in \R^{n-k+2}$, $k\in (n]$, at rate ${\tiny \binom{n}{k}}\lambda_{n,k}$, where 
		\begin{equation}
			\lambda_{n,k}\coloneqq \int_{[0,1]} z^{k-2}(1-z)^{n-k}\Lambda(\dd z), \label{defrateneutjump}
		\end{equation}
		with the convention that $0^0=1$ (in particular $\lambda_{n,2}=\Lambda(\{0\})+\int_{(0,1]} (1-z)^{n-2}\Lambda(\dd z)$).
		\item $D^{n,l}w\in \R^{n+l}$, $l\in (\kappa]$, at rate $n\beta_l$.
		\item $M^{n,\ell}_c w\in \R^{n-\ell+1}$, $\ell\in[n]$ and $c\in\{a,A\}$, at rate ${\tiny \binom{n}{\ell}}m_{n,\ell}^c+\1_{\{\ell=1\}}n\theta_c$, where 
		\begin{equation}
			m^a_{n,\ell}=\int_{(0,1)} z^{\ell-1}(1-z)^{n-\ell}\nu( \dd z)\quad \textrm{and}\quad m^A_{n,\ell}=\int_{(-1,0)} \lvert z\rvert^{\ell-1}(1-\lvert z\rvert )^{n-\ell}\nu( \dd z).\label{defrateposjumpmut}
		\end{equation}
		\item $S^{n,\ell}_cw \in \R^{n+\ell+1}$, $\ell\in[n]$ and $c\in\{a,A\}$, at rate ${\tiny \binom{n}{\ell}}\sigma_{n,\ell}^c$, where 
		\begin{equation}
			\sigma^a_{n,\ell}=\int_{(0,1)} z^{\ell-1}(1-z)^{n-\ell}\mu( \dd z)\quad \textrm{and}\quad \sigma^A_{n,\ell}=\int_{(-1,0)} \lvert z\rvert^{\ell-1}(1-\lvert z\rvert)^{n-\ell}\mu(\dd z).\label{defrateposjumpenv}
		\end{equation}
	\end{itemize}
	Define $L_r\coloneqq\dim(V_r)-1$,
	such that we can write $V_r=(V_r(0),\ldots,V_r(L_r))$.	
	Denote by $\P_{v}$ the probability measure associated with $V$ under the initial condition $V_0=v$, 
	and by $\E_{v}$ the associated expectation. 
\end{defi}

We refer to the process $L\coloneqq (L_r)_{r\geq 0}$ as the \emph{line-counting process}, because its transitions are reminiscent of the line-counting process of an ancestral selection graph; this connection will be made precise in Section~\ref{sec:asg}. 
The fact that the processes $L$ and $V$ are well-defined at all times (i.e. that they are \textit{conservative}) follows from Lemmas \ref{non-exp} and  \ref{condq}, respectively. Note that if $V_0 =v\in C_V(\xi)$, then $V_t\in C_V(v)\subset C_V(\xi)$ for all $t\geq 0$, and hence one can further reduce the state of space to $C_V(v)$.

The duality exposed in the following theorem formalises the connection between the $\Lambda$-Wright--Fisher process and the Bernstein coefficient process and generalises \cite[Thm. 2.14]{Cordero2022} to our setting. This result does not require~\eqref{Assump} to hold and is proved in Section~\ref{sec:lcp}.
\begin{theo}[Bernstein duality]\label{thm:Bernsteinduality}
	Let $X$ be the solution of~\eqref{eq:SDEWFP} and let $V$ be the Bernstein coefficient process. For any $x\in[0,1]$ and $v=(v_i)_{i=0}^n\in\R^{n+1}$ we have, for~$t\geq0$,
	$$\E_x\left[\sum_{i=0}^n v_i \,\binom{n}{i} X_t^i(1-X_t)^{n-i}\right]=\E_v\left[\sum_{i=0}^{L_t} V_t(i)\,\binom{L_t}{i} x^i(1-x)^{L_t-i}\right].$$ 
\end{theo}
Note that the underlying \textit{duality function} is $H:[0,1]\times \R^{\infty}\to\Rb$ defined via 
\begin{align}
	H(x,w):=\sum_{i=0}^{\dim(w)-1}w_i \binom{\dim(w)-1}{i} x^i (1-x)^{\dim(w)-1-i}. \label{defh}
\end{align}
Next, we describe the long-term behavior of~$V$ and $X$ under the absence of mutation, that is, under~\eqref{eq:nomut}.
\begin{prop}[Long-term behavior of $V$ without mutation]\label{prop:asympV}
	Assume \eqref{eq:nomut} holds. Then, $(V_{r}(0))_{r\geq 0}$ and $(V_{r}(L_r))_{r\geq 0}$ are both constant. 
	Moreover, if in addition, \eqref{Assump} is satisfied, then the following holds.
	\begin{enumerate}  
		\item For every $a,b\in \R$, the Bernstein coefficient process $V$ on $C^0_V((a,b)^T)$ has a unique invariant probability measure $\mu^{a,b}$ with 
		support $C^0_V((a,b)^T) \subset \{w\in \Rb^{\infty} : \ w_0=a,\, w_{\dim(w)-1}=b \}$. 
		\item Let $V_{\infty}^{a,b}$ be a random variable with law~$\mu^{a,b}$. If $V_0=v \in \Rb^{\infty}$ with $v_0=a$ and $v_{\dim(v)-1}=b$, then $C^0_V((a,b)^T) \subset C^0_V(v)$ and $V_r$ converges in distribution to $V_{\infty}^{a,b}$ as $r\to\infty$.
	\end{enumerate}
\end{prop}

This result is proved analogously to \cite[Prop. 2.24]{Cordero2022}, which deals with the situation without mutation and coordinated selection. 
However, since the proof is short and informative, we provide it in Section~\ref{sec:propB} for the sake of completeness. 

\begin{theo}[Fixation/Extinction for $X$ without mutation]\label{fixext}
	Assume \eqref{Assump} and \eqref{eq:nomut} hold. 
	Let $X_0=x$ for some $x\in(0,1)$. Then $X_t$ converges almost surely to a $\{0,1\}$-valued random variable $X_\infty$.
	Moreover,
	\begin{align}
		\P_x(X_\infty=1)=\E[H(x,V^{0,1}_{\infty})]\in(0,1), \label{exprfivproba}
	\end{align}
	where the random variable $V^{0,1}_{\infty}$ is as in Proposition \ref{prop:asympV}. 
\end{theo}
This result is proved in Section~\ref{sec:prooffixation}. Finally, we describe the long-term behavior of~$V$ and $X$ in the presence of bidirectional mutation, that is, under~\eqref{eq:bimut}.

\begin{prop}[Absorption of $V$ under bidirectional mutation]\label{prop:absV}
	Assume \eqref{Assump} and~\eqref{eq:bimut} hold. Then, for any $v\in \R^\infty$, $\P_{v}$-almost surely $V$ is absorbed at a (random) state $V_\infty\in C_V(v) \cap \Rb^1$ in finite time.	
\end{prop}
The proposition is proved in Section~\ref{sec:propB}. The following theorem, which is proved in Section~\ref{sec:proofstationary}, shows that then the process $X$ has a unique stationary distribution, which can be expressed via the absorption probabilities of the process $V$. We denote by $U_\infty$ the unique entry of the vector $V_\infty\in\Rb^1$ from Proposition \ref{prop:absV}, i.e. $V_\infty=(U_\infty)$. 

\begin{theo}[Ergodicity of $X$ under bidirectional mutation]\label{thstadist}
	Assume \eqref{Assump} and~\eqref{eq:bimut} hold. Then, for any $x \in [0,1]$, $\P_x(X_t \underset{t \rightarrow \infty}{\longrightarrow} 1)=\P_x(X_t \underset{t \rightarrow \infty}{\longrightarrow} 0)=0$ and $X$ has a unique stationary distribution $\mathcal{L}$ that is characterized by its moments
	\begin{align}
		\forall n \in \N_0, \quad\rho_n := \int_{[0,1]} x^n \mathcal{L}(\dd x)=\E_{e_n}[U_\infty], \label{momentsloiinvar}
	\end{align}
	where $e_n$ denotes the $(n+1)$th unit vector in $\R^{n+1}$ and $U_{\infty}$ is as above.
	Moreover, for any initial distribution in $[0,1]$, the law of $X_t$ converges weakly to $\mathcal{L}$ as $t\to\infty$. 
\end{theo}

The following proposition, proved in Section~\ref{sec:proofstationary}, provides the recursive relations satisfied by the moments $\rho_n$ of the stationary distribution $\mathcal{L}$ from Theorem \ref{thstadist}. 
\begin{prop} [Moment recursions]\label{recursionmoments}
	Assume \eqref{Assump} and~\eqref{eq:bimut} hold. 
	Define for any $n \geq 1$,
	\begin{align*}
		\alpha_n \coloneqq & \frac{n(n-1)}{2} \Lambda(\{0\}) + \int_{(0,1]} (1-(1-r)^{n} -nr(1-r)^{n-1})r^{-2}\Lambda(\dd r) \\
		+ & \int_{(-1,1)} (1-(1-|r|)^{n}) |r|^{-1}(\mu+\nu)( \dd r) + n \bigg( \theta + \sum_{\ell=2}^{\kappa}\beta_\ell \bigg). 
	\end{align*}
	Moreover, set for $k \in [n-1]_0$, 
	\begin{align*}
		\alpha_{n,k} \coloneqq {\binom{n}{k}} \left ( m^a_{n,n-k} + \theta_a \mathds{1}_{k=n-1} \right ) + \mathds{1}_{k \geq 1} {\binom{n}{k-1}} \lambda_{n,n-k+1}, 
	\end{align*} 
	and for $k \in \{n,\dots,(n+\kappa-1)\vee2n\}$, 
	
	\begin{align*}
		\alpha_{n,k} \coloneqq& \mathds{1}_{k \leq n+\kappa-1} \bigg ( \sum_{\ell=2 \vee(k+1-n)}^{\kappa}n \beta_\ell \sum_{i=n}^{k} \frac{(-1)^{k-i} p_{i+1-n}^{(\ell)} l!}{(n+l-1-k)!(k-i)!(i+1-n)! } \bigg) \\
		&+  \mathds{1}_{k \leq 2n} \Bigg( \mathds{1}_{k \geq n+1} \binom{n}{k-n} \sigma_{n,k-n}^A + \sum_{\ell=1\vee(k-n)}^{n} \binom{n}{\ell} \sigma_{n,\ell}^a \sum_{i=n}^{k} \frac{(-1)^{k-i} 2^{n+\ell-i} l!}{(n+l-k)!(k-i)!(i-n)! } \Bigg). 
	\end{align*} 
	Then we have
	\begin{align}
		\alpha_n \rho_n & = \sum_{k=0}^{(n+\kappa-1)\vee2n} \alpha_{n,k} \rho_k. \label{linrelmoments}
	\end{align}
\end{prop}

\begin{remark}[On the role of Condition \eqref{Assump}]\label{oncondition}
	It was shown in \cite{Cordero2022} that any polynomial vanishing at $0$ and $1$ can be expressed as $d_{\sel}$, 
	and that there are then infinitely many choices of $\sel$. 
	Thus, Condition \eqref{Assump} is not completely determined by the jump-diffusion $X$, but depends on the choice of $\sel$; some of them may satisfy \eqref{Assump} and some of them may not. 
	For the purposes of this paper, the choice of $\sel$ and the optimality of Condition~\eqref{Assump} are not crucial, because if~\eqref{Assump} is not satisfied for a given choice of $\sel$, then Condition~\eqref{condla} (in Section~\ref{sec:discussionprevresults}) is automatically satisfied, and thus the long-term behavior of $X$ can be deduced from the results in~\cite{cordhumvech2022}. This also means that the same process $X$ can have multiple Bernstein duals (one for each choice of $\sel$); see~\cite{Cordero2022} for a discussion on minimal Bernstein duals of a given process $X$. 
	We will see in Lemma~\ref{lc-pos_abs} that Condition \eqref{Assump} implies the positive recurrence (resp. absorption at $0$) of the process~$L$ in the case without (resp. with) mutation. Depending on the choice of $\sel$, the process $L$ may or may not be positive recurrent (resp. absorbed at $0$); if $L$ is null-recurrent or transient, then \eqref{Assump} does not hold, and as mentioned above, the behavior of $X$ follows from the results in \cite{cordhumvech2022}.
	
	Similarly, since the line-counting process $L$ does not depend on the choice of the parameter $p$, its positive recurrence yields the long-term behavior of infinitely many $\Lambda$-Wright-Fisher processes (sharing all parameters but~$p$).
	
\end{remark}

\subsection{Discussion of related results}\label{sec:discussionprevresults}
In this section, we briefly discuss how our findings relate to the existing literature. To this end, three specific forms of selective drift will be discussed.

\textit{Genic (negative) selection}: 
It is probably the most classic way to incorporate selection into the Wright-Fisher process~\cite{Kimura1962}.
It can be recovered by setting $\kappa=2$, $p_1^{(2)}=0$, $\beta_2>0$ and
$$d_{\sel}(x)=-sx(1-x),\quad x\in[0,1].$$
Here, type $a$ plays the role of unfit or deleterious type.

\textit{Fittest-type win selection}: Discussed, for example, in~\cite{GS18}, this type of selection is recovered by fixing $\kappa\geq 2$, and for $l\in(\kappa]$ and $i\in[l]$, setting $\beta_l\geq 0$  and $p_i^{(l)}=\delta_{i,l}$. This yields
$$d_{\sel}(x)=-x(1-x)\sum_{i=0}^{\kappa-2}\left(\sum_{l=i+2}^\kappa \beta_l\right)x^i ,\quad x\in[0,1]$$ 
Here type $a$ plays the role of an unfit or deleterious type.

\textit{Balancing selection}: The type that is positively selected depends on the current type composition (e.g.~\cite{Neuhauser1999}). It corresponds to the case $\kappa=3$, $\beta_2=0$ , $\beta_3=s>0 $, $p_1^{(3)}=1$ and $p_2^{(3)}=0$, hence
$$d_{\sel}(x)=2sx(1-x)(1-2x),\quad x\in[0,1].$$  

We conclude the discussion with a remark on the form of \textit{selection coordinated by an environment}. 
The last summand in~\eqref{eq:SDEWFP} corresponds to coordinated selection, it may be interpreted as strong genic selection during infinitesimally short periods of time.
These events translate to simultaneous binary branchings in the dual process. 
More complex forms of coordinated selection are conceivable, such as frequency-dependent coordinated selection, which we expect to result in a backward process with more general branching structure. 
One reason we did not explore this further is that the formal dual process becomes more involved, which obscures the underlying construction. 
We refer to~\cite{casa19e} for a form of coordinated selection that features a more general branching structure in the backward process, but with the requirement that the same type is favored in all selection events.

\subsubsection{Previous results in absence of mutation} \label{sec:discussnomutation}

To date, moment duality has played a crucial role in the derivation of fixation/extinction probabilities for specific cases of unidirectional selection. For example, the case of genic selection has been treated in~\cite{foucart2013impact}, and was extended to fittest-type win selection in~\cite{GS18} (all assuming $\mu=0$). Coordinated selection that is always favorable to type-$A$ individuals was incorporated in~\cite{casa19e} for fittest-type win selection (see also~\cite{CorderoVechambre} for the case of genic selection with $\Lambda(\dd z)=2\delta_0(\dd z)$). 
Among other findings, these works established that fixation and extinction of the type-$a$ subpopulation occurs with positive probability if and only if
\begin{equation}\label{assu-2}
	\sum_{\ell=2}^\kappa \beta_\ell (\ell-1)+\int_{(-1,0)}\log(1-z)\frac{\mu(\dd z)}{\lvert z\rvert}<\int_{[0,1]}|\log(1-z)|\frac{\Lambda(\dd z)}{z^2}.
\end{equation}
However, in the case of bidirectional selection, i.e. under \eqref{eq:bisel} (for example, under balancing selection),  
moment duality is not expected to hold (not for $X$ nor for $1-X$), 
and hence other approaches were needed. An alternative approach based on combinatorial properties of the underlying ancestral structures was proposed in \cite{Vec2023} in the setting of coordinated bidirectional selection and genic selection (restricted to  $\Lambda=\delta_{{0}}$ and $\int_{(-1,1)}|z|^{-1}\mu(\dd z)<\infty$).
Bernstein duality, which contains moment duality as a special case, was introduced in~\cite{Cordero2022} to treat frequency-dependent (bidirectional) selection (however, assuming $\mu=0$). 
It was shown that under~\eqref{assu-2}, extinction and fixation occur with positive probability. 
In particular, Theorem~\ref{fixext} coincides with \cite[Prop. 2.27]{Cordero2022}.

In the absence of mutation, \eqref{eq:nomut}, a classification of the possible boundary behavior of $X$ was derived in~\cite{cordhumvech2022} using Siegmund duality: almost sure fixation, almost sure extinction, fixation and extinction with positive probabilities and coexistence.
The advantage of Siegmund duality is that it holds under very mild conditions (monotonicity with respect to the initial value). In particular, the analysis in~\cite{cordhumvech2022} covers all the evolutionary forces considered in the present paper (for the part without mutation), but relies on the assumption
\begin{equation}\label{condla}
	\Lambda(\{0,1\})=0\quad\textrm{and}\quad \int_{(0,1)}\frac{\Lambda(\dd z)}{z}<\infty.
\end{equation}
The condition can be thought of as ensuring that 1) the weight of small neutral reproductions is not too strong and that 
2) no single neutral reproduction replaces the entire population. 
To see how Theorem~\ref{fixext} fits into the setting of \cite{cordhumvech2022}, let us set $\sigma$ as the polynomial satisfying $x(1-x)\sigma(x)=d_{\sel}(x)$ for all $x\in[0,1]$; 
$\sigma$ is well-defined, because $d_{\sel}$ is a polynomial vanishing at $0$ and $1$. 
%Also let us set $\hat{\mu}(\dd z)=|z|^{-1}\mu(\dd z)$.
According to~\cite[Thm. 2.1(2)]{cordhumvech2022} $X_t$ converges a.s. to a $\{0,1\}$-valued random variable $X_\infty$ with $\P_x(X_\infty=1)\in(0,1)$ 
if
\begin{align}
	\sigma(0)+\int_{(-1,1)}\log(1+z)\frac{\mu(\dd z)}{|z|}&<\int_{[0,1]}|\log(1-z)|\frac{\Lambda(\dd z)}{z^2} \label{c0},\\
	-\sigma(1)+\int_{(-1,1)}\log(1-z)\frac{\mu(\dd z)}{|z|}&<\int_{[0,1]}|\log(1-z)|\frac{\Lambda(\dd z)}{z^2}\label{c1}.
\end{align}
Note that if \eqref{condla} is not satisfied, then
\eqref{Assump} is trivially satisfied (the right-hand side then is~$\infty$). 
Therefore, Theorem~\ref{fixext} implies that the statement of \cite[Thm. 2.1(2)]{cordhumvech2022} also holds when~\eqref{condla} is not satisfied. 
This closes the gap in the boundary classification \cite[Thm. 2.1]{cordhumvech2022}.

The Siegmund and Bernstein duality approaches are successful under contrasting conditions. 
The former is effective when neutral reproduction is weak, while the latter is suitable when it is strong.

\subsubsection{Previous results that consider mutation}\label{sec:discussmutation}

The convergence of $X$ towards its unique stationary distribution is known for some subsets of model parameters if $\theta_a,\theta_A>0$ and $\nu=0$. 
For the classical Wright--Fisher diffusion with bidirectional mutation and frequency-dependent selection, which corresponds to $\Lambda(\dd  z)=2\delta_0(\dd z)$, $\mu=0$,  $\nu=0$ and $\theta_a,\theta_A>0$,
the asymptotic behavior can be derived via diffusion theory, see e.g. \cite{Ta07}. 
An alternative derivation builds on the concept of moment duality, see~\cite{BW18} for the case of genic selection  and~\cite{BEH21} for the case of fittest-type win selection. 

This approach has the advantage of being more flexible at incorporating jumps, while providing a representation of the moments of the stationary distribution in terms of the absorption probabilities of a suitable Markov chain; the disadvantage of the method is that it imposes rather strong restrictions on some of the parameters (for example, it excludes the case of bidirectional selection). 
Two important contributions in this direction are \cite{CM19}, which considers genic selection, large neutral offspring events (i.e. $\Lambda\in\Ms_f([0,1])$) and no coordinated selection ($\mu=0$), and \cite{CorderoVechambre}, which, in the context of genic selection, assumes that $\Lambda(\dd  z)=2\delta_0(\dd z)$ and that coordinated selection is triggered by an environment that is always favorable to type $A$ individuals (i.e. $\mu$ is supported in $(-1,0)$). All these works assume that $\nu=0$, which means that they do not incorporate coordinated mutation. 
However, it is known from \cite{GKT} that moment duality still holds in the presence of coordinated mutation, and hence, one could expect that this can be used in the standard way to derive the asymptotic behavior of $X$. 
Since the classical proof of moment duality operates additively on the individual components of the generator, it seems reasonable to posit that moment duality holds, subject only to the constraints that $\mu$ is supported on the interval $(-1,0)$ and that selection is of the fittest-type win. Under condition \eqref{Assump}, Theorem~\ref{fixext} extends the above mentioned results to cover also bidirectional selection (i.e. cases where \eqref{eq:bisel} holds).

\section{A genealogical perspective on the Bernstein coefficient process}\label{sec:Bernstein}

The purpose of this section is twofold. 
On the one hand, we aim to provide intuition about the dynamics of  $X$ by describing an individual-based dynamics that approximates the model.
On the other hand, we hope to outline a framework for constructing the Bernstein dual process based on this individual-based dynamic, enabling its application to other dynamics.
Except Definitions~\ref{defl} and~\ref{def:bramutoperators}, the section can be considered optional as it only aims to provides intuition and is not needed for the proofs.

\subsection{A Moran model counterpart and the ancestral picture}\label{sec:moran}
It is well-known that SDEs of the form \eqref{eq:SDEWFP} arise as large population limits of constant size Wright--Fisher and Moran models. 
We now describe a Moran model suitable for our context, and explain how one can derive the corresponding ancestral structure,
which forms the basis for the genealogical perspective to the Bernstein coefficient process.

\subsubsection{The Moran model}
Consider a population of constant size $K$ consisting of two types of individuals: $a$ and $A$. Let $X_t^{(K)}$ denote the fraction of type-$a$ individuals present in the population at time $t$. 
The population reproduces neutrally, selectively or in a coordinated way (triggered by an environment), and is subject to mutation, which can occur individually or in a coordinated manner. 
On the basis of the parameters and the Poisson measures appearing in \eqref{eq:SDEWFP}, the parameters of the Moran model dynamic can be chosen to be
${\vec{\theta}}^{(K)}\coloneqq {\vec{\theta}}/K$, $\beta^{(K)}\coloneqq \beta/K$, $\tilde{N}^{(K)}([0,t] \times \cdot \times \cdot) \coloneqq \tilde{N} ([0,t/K] \times \cdot \times \cdot)$, $M^{(K)}([0,t] \times \cdot) \coloneqq M([0,t/K] \times \cdot)$ and $S^{(K)}([0,t] \times \cdot) \coloneqq S([0,t/K] \times \cdot)$ (however, this choice aims at simplicity, but is not the only way to code the model). 
The population then evolves via the following mechanisms (with the labeling in line with the labeling of the building blocks of~\eqref{eq:SDEWFP}).

\medskip

\noindent (a)-i \emph{Single neutral offspring}: each individual gives birth independently at rate $\Lambda(\{0\})/2$, producing one offspring that inherits the parental type and replaces a randomly chosen individual in the population.\\
(a)-ii \emph{Large neutral offspring}: at points $(t,z,u)$ of $\tilde{N}^{(K)}$ {each individual is killed, independently of others, with probability $z$} and is replaced by the offspring of one individual in the population, which is chosen uniformly at random among the type~$a$ (resp.~$A$) individuals if $X_{t-}^{(K)}\leq u$ (resp. $X_{t-}^{(K)}>u$).\\
(b) \emph{Single selective offspring} is the result of $\ell$\emph{-interactions}, $\ell\in(\kappa]$. This means that each individual is independently selected at rate $\beta_\ell^{(K)}$ to gather a group of $\ell$ randomly chosen individuals (with the selected individual being part of the~$\ell$-group). 
If the group consists of $i$ type-$a$ individuals and $\ell-i$ type-$A$ individuals, one individual is chosen uniformly at random from the type $a$ (resp. $A$) population with probability $p_{i}^{(\ell)}$ (resp. $1-p_i^{(\ell)}$) to produce one single offspring that inherits the parental type and replaces the founder of the group. \\
(c) \emph{Coordinated selective offspring} is produced whenever there is a jump of the environment, which is encoded by ${S}^{(K)}$.
More precisely, at the points $(t,z)$ of ${S}^{(K)}$ with $z>0$ (resp. $z<0$) each individual of type $a$ (resp. $A$) independently produces with probability $\lvert z \rvert$ a single offspring.
This offspring inherits the parental type, and the resulting group of newborns replaces a uniformly chosen group of the same size from the population. \\
(d) \emph{Individual mutation} affects each individual independently at constant rate. Specifically,  each individual mutates to type $a$ (resp. $A$) at rate $\theta_a^{(K)}$ (resp. $\theta_A^{(K)}$).\\ 
(e) \emph{Coordinated mutation} is driven by the measure $M^{(K)}$: at the points $(t,z)$ of ${M}^{(K)}$ with $z>0$ (resp. $z<0$) each individual independently of each other, mutates with probability $\lvert z\rvert$ to type $a$ (resp. $A$).

\smallskip

It can be proved that a time-rescaled Moran model $(X_{Kt}^{(K)})_{t\geq 0}$ converges in distribution as $K\to\infty$ to the solution {$(X,\P_x)$} of \eqref{eq:SDEWFP}, provided that the initial conditions $X_0^{(K)}$ converge to {$x\in[0,1]$}. 
This has been shown in several special cases of~\eqref{eq:SDEWFP} via convergence of the corresponding generators, see e.g. \cite{GS18} (for the case $p_{i}^{(\ell)}=\delta_{i,\ell}$, $\mu=\nu=0$, and {$\vec{\theta}=0$}), \cite{Cordero2022} (for the case $\mu=\nu=0$ and {$\vec{\theta}=0$}), \cite{CorderoVechambre} (for the case $\Lambda=\delta_{\{0\}}$, $\mu((-1,0))=0$, and $\nu=0$); see also \cite{BCM19} for a rather general convergence result for discrete-time Wright--Fisher models. 
Since our focus is on analyzing the limit process rather than deriving it,  we omit the formal argument.

\subsubsection{Graphical representation}
The evolution of the Moran model can be illustrated through a graphical representation, a common approach in the study of interacting particle systems. 
This representation reveals an embedded ancestral structure that is useful for identifying the Bernstein dual process. 
We will now clarify how this representation manifests in our specific context.

The graphical representation is the diagram $\Ds\coloneqq [0,\infty)\times [K]$ together with graphical elements randomly superposed in a type-agnostic manner on $\Ds$ coding the events driving the population's evolution.
We identify each individual in the population with a label in $[K]$; the individual with label~$i$ at time $t$ corresponds to the point $(t,i)\in\Ds$. For $i\in[K]$, we call $[0,\infty)\times \{i\}\subset \Ds$ the \emph{$i$th line}. 
An individual's lifetime is represented by the horizontal segment $[t_0,t_\dagger)\times{\{i\}}$,
where $i$ is its label, and~$t_0$ and~$t_\dagger$ denote its birth and death times, respectively.
The population evolves using the graphical elements explained in the following, 
where the categories should be compared also with the one in the previous section.

\smallskip

\noindent (a) \emph{Neutral offspring events} are depicted as vertical arrows with the parental line at the tail and the offspring line at the tip (multiple arrows sharing a tail in the case of large offspring). We call them neutral arrow(head)s.\\
(b) \emph{Single selective offspring} events are coded with a set of open-headed arrows, with the tips pointing to the individual to be replaced and the tails at the other potential parents, we refer to them as selective arrow(head)s. 
In particular, in an $\ell$-interaction, these are $\ell-1$ arrows.\\
(c) \emph{Coordinated selection} is depicted by arrows with a white (resp. black) diamond at the tail if the corresponding environment triggering the reproduction favors type $a$ (resp. $A$); we call them coordinated arrow(head)s. 
Specifically, arrows are placed on $\Ds$ when there is a jump in~$S^{(K)}$. 
For a point $(t,z)$ of~$S^{(K)}$ with $z>0$ (resp. $z<0$), 
choose $k\sim\bindist{K}{\lvert z\rvert}$ ordered pairs of lines, and draw an arrow from the first line to the second for each pair. 
The tails of the arrows are represented by white diamonds (resp. black diamonds). \\
(d)\&(e) \emph{Mutation} to type $a$ (resp. $A$) is triggered by the graphical element of a white (resp. black) \emph{circle}.

\smallskip

The graphical elements are superposed on~$\Ds$  according to the dynamics described in the previous section, but ignoring any information about the types of the individuals. See Fig.~\ref{fig:Diagram} for an illustration, where forward time runs from left to right.

\begin{figure}[t!]
	\centering
	\scalebox{0.65}{\begin{tikzpicture}
			
			\draw[dotted] (9.5,-1) -- (9.5,4.5);
			\draw[dotted] (8,-1) -- (8,4.5);
			\draw[dotted] (6,-1) -- (6,4.5);
			\draw[dotted] (4,-1) -- (4,4.5);
			\draw[dotted] (2,-1) -- (2,4.5);
			
			\draw[dotted] (1,-1) -- (1,4.5);
			\draw[dotted] (0,-1) -- (0,4.5);
			\draw[dotted] (-1,-1) -- (-1,4.5);
			\draw[dotted] (-2,-1) -- (-2,4.5);
			\draw[dotted] (-3,-1) -- (-3,4.5);
			\draw[dotted] (-4,-1) -- (-4,4.5);
			
			\node[below] at (9.5,5.2) {$0$};
			\node[below] at (8.75,5.2) {\color{red}$3$\color{black}};			\node[below] at (8,5.2) {$r_1$};
			\node[below] at (7,5.2) {\color{red}$4$\color{black}};			\node[below] at (6,5.2) {$r_2$};
			\node[below] at (5,5.2) {\color{red}$7$\color{black}};						\node[below] at (4,5.2) {$r_3$};
			\node[below] at (3,5.2) {\color{red}$6$\color{black}};			\node[below] at (2,5.2) {$r_4$};
			\node[below] at (1.5,5.2) {\color{red}$4$\color{black}};			\node[below] at (1,5.2) {$r_5$};
			\node[below] at (0.5,5.2) {\color{red}$2$\color{black}};			\node[below] at (0,5.2) {$r_6$};
			\node[below] at (-0.5,5.2) {\color{red}$4$\color{black}};						\node[below] at (-1,5.2) {$r_7$};
			\node[below] at (-1.5,5.2) {\color{red}$6$\color{black}};			\node[below] at (-2,5.2) {$r_8$};
			\node[below] at (-2.5,5.2) {\color{red}$9$\color{black}};						\node[below] at (-3,5.2) {$r_9$};
			\node[below] at (-3.5,5.2) {\color{red}$5$\color{black}};						\node[below] at (-4,5.2) {$T$};
			
			\node[below] at (9.5,-1.3) {$T$};
			\node[below] at (8,-1.3) {$t_9$};
			\node[below] at (6,-1.3) {$t_8$};
			\node[below] at (4,-1.3) {$t_7$};
			\node[below] at (2,-1.3) {$t_6$};
			\node[below] at (1,-1.3) {$t_5$};
			\node[below] at (0,-1.3) {$t_4$};
			\node[below] at (-1,-1.3) {$t_3$};
			\node[below] at (-2,-1.3) {$t_2$};
			\node[below] at (-3,-1.3) {$t_1$};
			\node[below] at (-4,-1.3) {$0$};
			
			%%%noasg
			\draw[opacity=0.2, line width=.5mm]  (9.5,4.5) -- (-4,4.5);
			\draw[opacity=0.2, line width=.5mm]  (9.5,4) -- (-4,4);
			\draw[opacity=0.2, line width=.5mm]  (9.5,3.5) -- (-4,3.5);
			\draw[opacity=0.2, line width=.5mm]  (9.5,3) -- (-4,3);
			\draw[opacity=0.2, line width=.5mm]  (9.5,2.5) -- (-4,2.5);
			\draw[opacity=0.2, line width=.5mm]  (9.5,2) -- (-4,2);
			\draw[opacity=0.2, line width=.5mm]  (9.5,1.5) -- (-4,1.5);
			\draw[opacity=0.2, line width=.5mm]  (9.5,1) -- (-4,1);
			\draw[opacity=0.2, line width=.5mm]  (9.5,0.5) -- (-4,0.5);
			\draw[opacity=0.2, line width=.5mm]  (9.5,0) -- (-4,0);
			\draw[opacity=0.2, line width=.5mm]  (9.5,-0.5) -- (-4,-0.5);
			\draw[opacity=0.2, line width=.5mm]  (9.5,-1) -- (-4,-1);
			
			% asg Horizontal
			\draw[color=red,opacity=1, line width=.5mm]  (9.5,3.5) -- (2,3.5);
			
			\draw[color=red,opacity=1, line width=.5mm] (8,3) -- (2,3);
			
			\draw[color=red,opacity=1, line width=.5mm] (6,4) -- (1,4);
			\draw[color=red,opacity=1, line width=.5mm] (6,2.5) -- (2,2.5);
			\draw[color=red,opacity=1, line width=.5mm]  (6,2) -- (-2.87,2);
			\draw[color=red, opacity=1, line width=.5mm] (9.5,1) -- (4,1) ;
			\draw[color=red,opacity=1, line width=.5mm] (4,0.5) -- (3.7,.5) ;
			\draw[color=red, opacity=1, line width=.5mm] (3.7,0.5) -- (1,.5) ;
			\draw[color=red,opacity=1, line width=.5mm]  (9.5,0) -- (4,0) ;
			\draw[color=red,opacity=1, line width=.5mm] (2,3) -- (-2.87,3);
			
			\draw[color=red,opacity=1, line width=.5mm] (0,3.5) -- (-4,3.5);
			\draw[color=red,opacity=1, line width=.5mm] (0,1.5) -- (-4,1.5);
			\draw[color=red,opacity=1, line width=.5mm] (-1,4) -- (-4,4);
			\draw[color=red,opacity=1, line width=.5mm] (-1,2.5) -- (-4,2.5);
			\draw[color=red,opacity=1, line width=.5mm] (-2,1) -- (-2.87,1);
			\draw[color=red,opacity=1, line width=.5mm] (-2,0.5) -- (-4,0.5);
			\draw[color=red,opacity=1, line width=.5mm] (-2,0) -- (-2.87,0);

			% asg vertical
			
			\draw[opacity=0.2, line width=.5mm, -{Stealth[length=2mm,width=2mm]}] (-1,0.5) -- (-1,1);
			\draw[opacity=0.2, line width=.5mm, -{Stealth[length=2mm,width=2mm]}] (-3.6,1) -- (-3.6,0);
			
			\draw[opacity=0.2, line width=.5mm, -{Stealth[length=2mm,width=2mm]}] (7.4,2.5) -- (7.4,2);
			\draw[opacity=0.2, line width=.5mm, -{Stealth[length=2mm,width=2mm]}] (7.4,2.5) .. controls (7.2,1.7) .. (7.4,1.5);
			\draw[opacity=0.2, line width=.5mm, -{Stealth[length=2mm,width=2mm]}] (7.4,2.5) .. controls (7.2,1) .. (7.4,-0.5);
			
			\draw[opacity=1, very thick, line width=.5mm, {Latex[length=3mm,width=3mm,open]}-,opacity=0.7] (8,3.5) -- (8,3);
			
			\draw[opacity=1, line width=.5mm, {Latex[length=3mm,width=3mm,open]}-, opacity=0.7] (6,3) -- (6,2.5);
			\draw[line width=.5mm, {Latex[length=3mm,width=3mm,open]}-, opacity=0.7] (6,3) .. controls (6.3,2.5) and (6.2,2.5) .. (6,2);
			\draw[line width=.5mm, {Latex[length=3mm,width=3mm,open]}-, opacity=0.7] (6,3) .. controls (6.1,3.5) and (6.1,3.5) .. (6,4);

			\draw[line width=.5mm, {Stealth[length=2mm,width=2mm]}-] (2,3.5) .. controls (1.8,3.2) ..(2,3);
			\draw[line width=.5mm, {Stealth[length=2mm,width=2mm]}-] (2,2.5) .. controls (1.8,2.7) ..(2,3);

			\draw[opacity=1, line width=.5mm, -{Stealth[length=2mm,width=2mm]}] (4,.5) .. controls (3.8,0.2) .. (4,0);
			\draw[opacity=1, line width=.5mm, {Stealth[length=2mm,width=2mm]}-] (4,1) .. controls (3.8,0.7) ..(4,0.5);
			
			\draw[line width=.5mm, {Stealth[length=2mm,width=2mm]}-] (0,3) -- (0,3.5);
			\draw[line width=.5mm, {Stealth[length=2mm,width=2mm]}-] (0,2) -- (0,1.5);
			\draw[line width=.5mm, {Stealth[length=2mm,width=2mm]}-] (-1,3.5) -- (-1,4);
			\draw[line width=.5mm, {Stealth[length=2mm,width=2mm]}-] (-1,2) -- (-1,2.5);
			
			\draw[line width=.5mm, -{Latex[length=3mm,width=3mm,open]},opacity=1] (-2,0) .. controls (-1.7,.6) and (-1.3,.9) .. (-2,1.5);	
			\draw[line width=.5mm, -{Latex[length=3mm,width=3mm,open]},opacity=1] (-2,0.5) .. controls (-1.8,1) and (-1.8,1) .. (-2,1.5);			
			\draw[line width=.5mm, -{Latex[length=3mm,width=3mm,open]},opacity=1] (-2,1) -- (-2,1.5);
			
			\draw[line width=.3mm, -{Latex[length=3mm,width=2mm,open]},opacity=0.2] (3.2,4.5) .. controls (3.5,3) and (3.5,2) .. (3.2,1.5);	
			\draw[line width=.3mm, -{Latex[length=3mm,width=2mm,open]},opacity=.2] (3.2,-0.5) .. controls (3.5,1) and (3.5,1) .. (3.2,1.5);	
			
			\draw[line width=.3mm, {Stealth[length=3mm,width=2mm]}-,opacity=1] (2,4.7) -- (3,4.7);
			\draw[line width=.3mm, -{Stealth[length=3mm,width=2mm]},opacity=1] (2,-1.2) -- (3,-1.2);
			
			\begin{scope}[every node/.style={diamond,fill,thick,draw,scale=0.5}]
				\node (F) at (-1,4) {};
				\node (G) at (-1,2.5) {};
				\node[opacity=0.3] (F10) at (-1,0.5) {};
			\end{scope}
			
			\begin{scope}[every node/.style={circle,fill,thick,draw,scale=0.6}]
				\node (A) at (1,4) {};
				\node (B) at (1,.5) {};
				\node[opacity=0.3] (F1) at (2.2,-0.5) {};
				\node[opacity=0.3] (F2) at (8.4,4.5) {};
				\node[opacity=0.3] (F3) at (8.4,3) {};
				\node[opacity=0.3] (F4) at (8.4,2.5) {};
				\node[opacity=0.3] (F5) at (8.4,2) {};
				\node[opacity=0.3] (F6) at (8.4,-1) {};
			\end{scope}
			
			\begin{scope}[every node/.style={diamond,fill=white, thick,draw,scale=0.5}]
				\node (C) at (0,3.5) {};
				\node (D) at (0,1.5) {};
			\end{scope}
			
			%%%%%
			\begin{scope}[every node/.style={circle,thick,fill=white,scale=0.6}]
				\node[] at (5.3,1.5) {};	
				\node[] at (2.7,1) {};
				\node[] at (-0.3,0) {};			
			\end{scope} 
			
			\begin{scope}[every node/.style={circle,thick,fill=white,draw,scale=0.6}]
				
				\node (H) at (-3,3) {};
				\node (I) at (-3,2) {};
				\node (J) at (-3,0) {};
				\node (K) at (-3,1) {};
				\node[opacity=0.2] (F7) at (5.3,1.5) {};
				\node[opacity=0.2] (F8) at (2.7,1) {};
				\node[opacity=0.2] (F9) at (-0.3,0) {};
				
			\end{scope}
			
	\end{tikzpicture}}
	\caption{A graphical representation of the untyped Moran model (bold and semi-transparent elements) and the associated ASG (bold and red elements) to a sample of three individuals at forward time $t=T$. The red number between $r_i$ and $r_{i+1}$ represents the number of lines in the ASG in backward time interval $[r_i,r_{i+1})$.}
	\label{fig:Diagram}
\end{figure}

Once these graphical elements are known between times $0$ and $T$, we can allocate a type (that is, either~$a$ or~$A$) to each individual at time $t=0$, and propagate types forward in time along the lines of the diagram respecting the \emph{propagation rules}. The first such rule is that an individual can change its type only when its line encounters a circle or an arrowhead. The other rules are the following.

\smallskip

\noindent (a) The type of an individual after a neutral arrowhead is the type of the individual at the tail of the corresponding arrow.\\
(b) In an $\ell$-interaction involving exactly $i$ individuals of type $a$, the type of the individual after being hit by the $\ell-1$ arrows is $a$ with probability $p_i^{(\ell)}$ and~$A$ with probability $1-p_i^{(\ell)}$.\\
%The individual just after a selective arrowhead with $\ell$ tails, of which $i$ carry a type~$a$ individual, is of type $a$ with probability $p_i^{(\ell)}$ and of type~$A$ with probability $1-p_i^{(\ell)}$.\\
(c) The individual just after a coordinated arrowhead with a white (resp. black) diamond at the tail gets the type of the individual sitting at the tail if the latter has type $a$ (resp. $A$) and, otherwise, keeps its original type (i.e. the arrow is ignored). \\ 
(d)$\&$(e) The type of an individual after a white (resp. black) circle is $a$ (resp. $A$). 

\smallskip

We say that a line is of type~$a$ or~$A$ if the individual occupying that line is of that type. One can verify that there is a choice of rate parameters for the Poisson processes indicating the placement of the graphical elements such that the distribution of type~$a$ in the graphical representation is that of a Moran model.

\subsubsection{Moran model ancestral structure}
As mentioned before, the graphical construction relates the Moran model to an ancestral structure.
The latter is then a generalization of the ancestral selection graph (ASG) by Krone and Neuhauser \cite{KroNe97,NeKro97} 
(or more precisely of the killed-ASG \cite{BW18}). 
We now make precise how to obtain that structure.

The ancestral structure arises from picking a sample of individuals at time $T$ and tracing back the lines of relevant potential ancestors, those that can influence the types in the sample. 
It will be convenient to introduce a backward time, which is denoted by $r$ and satisfies $r = T-t$ for a forward time $t\in[0,T]$.
Starting from a sample of individuals,
arrows and circles in the diagram change the number of relevant potential ancestors and their respective locations backward in time as follows (it is again instructive to keep in mind the labeling of the previous section).

\smallskip

\noindent(a) If a group of lines encounters neutral arrowheads, we stop following them and we add to the ASG the line at the (common) tail of the arrows (i.e. a coalescence event takes place), since the line at the tail determines the types of all the lines at the tips. \\
(b) If a line at a given time encounters $\ell-1$ selective arrowheads, we add the lines at the corresponding tails to the ASG, because they are required to determine the type of the line at the arrowtips.\\
(c) If a group of lines encounters coordinated arrowheads, we augment the ASG by adding the lines at the diamond tails, as we need knowledge of their types to determine the parent (and thus they are relevant potential ancestors).\\
(d)\&(e) If a line encounters a circle, we stop following that line since its type is determined by the mutation, but we retain the circle at the end of the line for bookkeeping purposes.\\

Compare this also with Fig.~\ref{fig:Diagram} for an illustration. 
To determine the types in the sample, only the lines in the ASG at time $t=0$ ($r=T$) have to get assigned a type and this type can then propagate forward in time using the type propagation rules in the diagram $\Ds$ until time $t=T$ ($r=0$). 
\subsection{The asymptotic ancestral selection graph and the Bernstein dual}\label{sec:asg}
\subsubsection{Ancestral selection graph} 
The next step in the derivation of the dual process is to
associate an ASG to SDE \eqref{eq:SDEWFP}.
That object should be the limit as $K\to\infty$ of the ASG associated to the Moran model of size $K$. Making this rigorous would require a precise description of the ASG as a mathematical element of an appropriate topological space, which is beyond the scope of this paper. Instead, we derive the transitions of the limiting ASG in a heuristic way through a simple analysis of the rates of the Moran ASG as $K\to\infty$, after time-rescaling by $K$. This will suffice for our purposes, namely to motivate the definition of the Bernstein coefficient process (the formal relation to the forward evolution will be proved in Section \ref{sec:lcp}). 
Proceeding in this manner leads to the following branching-coalescing system.

Initially, there are~$n$ (ancestral) lines, each representing one individual in a sample taken at forward time~$T$.
The following transitions of the ancestral structure occur.

\smallskip

\noindent (a) \emph{Coalescences (neutral events).} For $n>1$ and $k\in (n]$ every group of $k$ lines among $n$ \emph{coalesces} into one at rate $\lambda_{n,k}$.\\
(b) \emph{Selective branching.} Every line independently splits into $\ell$ at rate~$\beta_\ell$, that is, there are $\ell-1$ additional lines.\\
(c) \emph{Coordinated branching.} For $c\in\{a,A\}$, and $\ell\in [n]$ every group of $\ell$ lines among $n$ is subject to a simultaneous $c$-branching at rate $\sigma^c_{n,\ell}$, i.e. each of the $\ell$ lines splits into two lines, one continuing and one incoming. The incoming lines start with a diamond; the diamond is white for $c=a$ and black for $c=A$. \\
(d) \emph{Single mutation.} For $c\in\{a,A\}$, every line, independently of each other, is subject to a type-$c$ mutation at rate $\theta_c$. The line ends in a circle; the circle is white for $c=a$ and black for $c=A$. \\
(e) \emph{Coordinated mutation.} For $c\in\{a,A\}$, and $\ell\in [n]$ every group of $\ell$ lines among $n$ is subject to a coordinated $c$-mutation at rate $m^c_{n,\ell}$.
Each of the $\ell$ lines ends in a circle; the circle is white for $c=a$ and black for $c=A$.\\

The branching-coalescing structure arising under the above dynamics will be referred to as \emph{Ancestral selection graph} (ASG).

Write $L_r$ for the number of lines present in the ASG at (backward) time $r$ (it will become clear later why the notation is as in Section~\ref{sec:statementmainresults} for $\dim{V_r}-1$). 
We refer to the process $L\coloneqq(L_r)_{r\geq 0}$ as the line-counting process of the ASG; it is a continuous-time Markov Chain on $\N_0$ with transitions from $n\in \N$ to

\smallskip

\noindent (a) $n-k+1$ for $k\in (n]$ at rate ${\tiny \binom{n}{k}}\lambda_{n,k}$ if $n>1$.\\
(b) $n+l-1$ for $l\in (\kappa]$  at rate $n\beta_l$.\\
(c) $n+\ell$ for $\ell\in [n]$ at rate ${\tiny \binom{n}{\ell}}\sigma_{n,\ell}$ with $\sigma_{n,\ell}\coloneqq\sigma_{n,\ell}^a+\sigma_{n,\ell}^A$. \\
(d)\&(e) $n-\ell$ for $\ell\in[n]$ at rate ${\tiny \binom{n}{\ell}}m_{n,\ell}+\1_{\{\ell=1\}}n\theta$ with
$m_{n,\ell}\coloneqq m_{n,\ell}^a+m_{n,\ell}^A$ and $\theta\coloneqq\theta_a+\theta_A. $

\smallskip

Note that if there is mutation, then $0$ is an absorbing state for $L$. 
If there is no mutation, the state $0$ is never reached if $L$ starts with at least one line, and hence, we can reduce the state space to $\Nb$.

Consider now an ASG that has evolved from time $r=0$ until (backward) time $r=T$. 
The individuals at time~$r=0$, from which lines start, are called \emph{sinks} and the lines at time~$r=T$ are called \emph{sources}. 
If we assign types to the sources of the ASG, the types propagate forward in time (from left to right) up to the sinks according to the same type propagation rules introduced for the graphical representation of the Moran model.

\subsubsection{Connecting the ancestral selection graph to the Bernstein coefficient process}
%Although the line-counting process of the ASG is formally defined, we haven't defined the ASG itself as a fully mathematical object (however, this can be done); 
%instead, it has been presented as a decorated counting process featuring graphical elements like diamonds, which represent a branching-coalescing structure.
%In this sense, the ASG is a rather unwieldy object. 
%However, in order to establish a formal relation between the process $X$ and its genealogy, 
%it suffices to express the distribution of the type composition in a sample of individuals at a given time in terms of a tractable functional of the ASG. 

Consider the task of computing the probability of sampling~$n$ individuals of type~$a$ at forward time~$T$.
When conditioning on $X_T$, this probability is $X_T^n$.
Alternatively, given a realization of the ASG in~$[0,T]$ associated with the sample, we can determine the probability that its sinks are of type $a$, given that types are assigned to the sources based on the initial type composition $(x,1-x)$. 
Determining this probability requires the full information encoded in the ASG, 
which is a rather cumbersome object. 
Instead, we compute these probabilities, not for a given ASG, but averaging over all its possible realizations compatible with a given line-counting process. 
This idea generalizes the strategy developed in \cite{Cordero2022} to our setting and paves the way to the Bernstein duality.
As a start, it will be convenient to augment the line-counting process by the information on the type of event triggering its last jump; we will distinguish between neutral ($\blacktriangle$), selective ($\triangle$), $a$-coordinated selection ($\diamond$), $A$-coordinated selection ($\blackdiamond$), $a$-mutations ($\circ$) and $A$-mutations ($\bullet$); we will use the symbol $*$ to indicate that no-events have occurred yet. This motivates the following definition.

\begin{defi}[Labeled line-counting process]\label{defl}
	Consider the set of symbols $\Sigma\coloneqq\{*,\blacktriangle, \triangle,\diamond,\blackdiamond,\circ,\bullet\}$. The labeled line-counting process is the $\Nb_0\times\Sigma$-valued continuous-time Markov chain $(L_r,c_r)_{r\geq 0}$  transitioning from $(n,c)\in \Nb_0\times\Sigma$ to:
	\begin{itemize}
		\item[(a)] $(n-k+1,\blacktriangle)$, for $k\in (n]$, at rate ${\tiny \binom{n}{k}}\lambda_{n,k}$ if $n>1$. 
		\item[(b)] $(n+\ell-1,\triangle)$, for $\ell\in (\kappa]$, at rate $n\beta_\ell$.
		\item[(c)] $(n+\ell,b)$, for $\ell\in [n]$ and $b\in\{\diamond,\blackdiamond\}$, at rate ${\tiny\binom{n}{\ell}}\sigma_{n,\ell}^b$.
		\item[(d)\&(e)] $(n-\ell,b)$, for $\ell\in[n]$ and $b\in\{\circ,\bullet\}$ at rate ${\tiny\binom{n}{\ell}} m_{n,\ell}^{b}+\1_{\{\ell=1\}}n\theta_b$.
	\end{itemize}
\end{defi}

Recall that the types at the sinks of an ASG represent the (unknown) types in a sample of the population. 
They are a (random) function of the ASG and the types at the sources. 
Now, consider the random function $x\mapsto P_T(x)$, where $P_T(x)$ denotes the conditional probability given $(L_r,c_r)_{r\in[0,T]}$, that all the sinks in the ASG are of type $a$ if each source in the ASG is of type $a$ or $A$ with probability $x$ or $1-x$, respectively. 
The function $P_T$ is referred to in \cite{Cordero2022} as the \emph{Ancestral Selection Polynomial} (not to be confused with the polynomial $d_{\vec{s}}$ modeling the contribution of frequency-dependent selection in SDE~\eqref{eq:SDEWFP}). 
The name comes from the fact that $P_T$ can be expressed as
$$P_T(x)=\sum_{i=0}^{L_T} V_T(i)\,\binom{L_T}{i} x^i(1-x)^{L_T-i},\qquad x\in[0,1],$$
where $V_T(i)$ then represents the conditional probability given $(L_r,c_r)_{r\in[0,T]}$, that all the sinks in the ASG are of type $a$, if $i$ sources are of type $a$, and $L_T-i$ sources are of type $A$. 
Thus, indeed $P_T$ is a polynomial. 
Moreover, its degree is at most $L_T$, and $V_T\coloneqq (V_T(i))_{i\in[L_T]_0}$ is its vector of coefficients in the Bernstein basis of polynomials of degree at most $L_T$.

The evolution of the coefficient process $(V_r)_{r\geq 0}$ is driven by corresponding transitions in the ASG.
How (a) coalescences and (b) selective branching affect the coefficient process is explained in \cite[Section 2.5]{Cordero2022}. 
Therefore, we here only focus on coordinated selection and mutation.

\smallskip

\noindent (c) \emph{The effect of coordinated selection.} Assume that $L_{r-} = n$, $L_r=n+\ell$ and $c_r=\diamond$, that is, a coordinated $a$-branching of size $\ell$ occurs at (backward) time $r$. 
Forward in time this means that $\ell$-pairs of lines experience a binary merger. 
If $i$ of the $n+\ell$ lines are of type $a$, let $R_{n,\ell,i}^a$ denote the sum of 1) the number of merging pairs containing at least one line of type $a$, and 2) the number of non-merging lines of type $a$. 
Then, by the propagation rules at the end of Section \ref{sec:asg}, at time $r-$, there are $R_{n,\ell,i}^a$ lines of type $a$. 
Note that $R_{n,\ell,i}^a\sim \mathrm{HP}(n+\ell,\ell,i)$, {where the distribution $\mathrm{HP}(\cdot,\cdot)$ is defined in Section~\ref{sec:prequel}}. 
Similarly, assume that $L_{r-} = n$, $L_r=n+\ell$ and $c_r=\blackdiamond$, that is, a coordinated $A$-branching of size $\ell$ occurs at (backward) time $r$. 
As before, this means that forward in time, $\ell$-pairs of lines experience a binary merger. 
If $i$ of the $n+\ell$ lines are of type $a$, let $R_{n,\ell,i}^A$ denote the sum of 1) the number of merging pairs where the two lines are of type $a$, and 2) the number of non-merging lines of type $a$. 
Then, at time $r-$, there are $R_{n,\ell,i}^A$ lines of type $a$. 
Note that $R_{n,\ell,i}^A\sim n-\mathrm{HP}(n+\ell,\ell,n+\ell-i)$.\\
(d)\&(e) \emph{The effect of mutation.} Assume that $L_{r-} = n$, $L_r=n-\ell$ and $c_r=\circ$, that is, an $a$-mutation event of size $\ell$ occurs at (backward) time $r$. 
If $i$ of the $n-\ell$ sources are of type $a$, then at time $r-$, $i+\ell$  of the $n$ sources have type $a$.
Similarly, assuming that $L_{r-} = n$, $L_r=n-\ell$ and $c_r=\bullet$, that is, an $A$-mutation event of size $\ell$ occurs at (backward) time $r$. 
If $i$ of the $n-\ell$ sources are of type $a$, then at time $r-$, $i$  of the $n$ sources have type $a$.

\smallskip

This motivates the definition of the following operators.

\begin{defi}[{Branching, mutation, coalescence operators}]\label{def:bramutoperators}
	For $n\in \N$, define the following linear operators.
	
	\begin{itemize}
		\item[(a)] Coalescence: For $k\in\,(n]$ and if $n>1$, define $C^{n,k}: \R^{n+1}\to\R^{n-k+2}$
		as $$C^{n,k}v\coloneqq\left(\frac{i}{n-k+1}v_{i+k-1}+\left(1-\frac{i}{n-k+1}\right)v_i\right)_{i=0}^{n-k+1}.$$
		\item[(b)] Selective branching: For $l\in\,(\kappa]$, define $D^{n,l}: \R^{n+1}\to\R^{n+l}$
		as  $$D^{n,l}v\coloneqq\left(\E[p_{K_i}^{(l)}v_{i+1-K_i}+(1-p_{K_i}^{(l)})v_{i-K_i}]\right)_{i=0}^{n+l-1},\quad\textrm{where } K_i\sim\mathrm{Hyp}(n+l-1,i,l). $$ 
		\item[(c)] Coordinated branching favoring type $b\in \{a,A\}$: For $\ell\in[n]$, let $S^{n,\ell}_b:\R^{n+1}\to \R^{n+\ell+1}$  as $$S^{n,\ell}_bv\coloneqq \left( \E\left[v_{R_{n,\ell,i}^{b}}\right]\right)_{i=0}^{n+\ell},$$ 
		where $R_{n,\ell,i}^{a}\sim \mathrm{HP}(n+\ell,\ell,i)$ and $R_{n,\ell,i}^{A}\sim n-\mathrm{HP}(n+\ell,\ell,n+\ell-i)$.
		
		\item[(d)\&(e)] Individual and coordinated mutation to type~$b\in \{a,A\}$: For $k\in\,[n]$, define $M^{n,k}_b: \R^{n+1}\to\R^{n-k+1}$ as $M^{n,k}_bv\coloneqq\left(v_{i+k}\right)_{i=0}^{n-k}$ if $b=a$, and $M^{n,k}_bv\coloneqq\left(v_{i}\right)_{i=0}^{n-k}$ if $b=A$.
	\end{itemize}
\end{defi}
It follows from this definition and Definition \ref{def:BCP} that $(\dim(V_r)-1)_{r\geq 0}$ and the line-counting process~$L$ have the same distribution, which justify the use of the same notation for both processes.

We provided above a probabilistic interpretation of $V$, as a process on $\cup_{n\in\N_0}[0,1]^{n+1}$, based on the ASG. However, the linear structure of the transitions displayed in Definition \ref{def:bramutoperators} allows one to extend $V$ to the larger set $\cup_{n\in\N_0}\R^{n+1}$, which we do in Definition \ref{def:BCP}. This yields insights on a greater variety of mixed moments of $X$ via the Bernstein duality (Theorem \ref{thm:Bernsteinduality}), and will be useful in future applications. 

\section{Properties of the Bernstein dual and proof of the duality}\label{sec:lcp}

In this section, we establish the duality and properties of the corresponding dual. 

\subsection{\texorpdfstring{First properties of the processes $L$ and $V$}{First properties of the processes L and V}} \label{secfirstprop}

Let us denote by ${\Bc}^L \in \R^{\Nb_0 \times \Nb_0}$ the rate matrix of the process $L$. More precisely, for any $i,j \in \Nb_0$ with $i \neq j$, ${\Bc}^L(i,j)$ is the total transition rate from $i$ to $j$ and, for any $i \in \Nb_0$, ${\Bc}^L(i,i):=-\sum_{j \in \Nb_0 \setminus \{i\}} {\Bc}^L(i,j)$. Note from Definition \ref{defl} that ${\Bc}^L(i,\cdot)$ has finite support for all $i \in \Nb_0$. For any function $g:\Nb_0\to\Rb$ we define ${\Bc}^Lg$ via ${\Bc}^Lg(i)=\sum_{j \in \Nb_0} {\Bc}^L(i,j) (g(j)-g(i))$. By Definition \ref{defl} this operator acts on functions $g:\Nb_0\to\Rb$ via ${\Bc}^Lg=({\Bc}^L_{\Lambda}+{\Bc}^L_{S}+{\Bc}^L_{D}+{\Bc}^L_{M, \nu}+{\Bc}^L_{M, \theta})g,$ where 
\begin{equation} \label{genl}
	\begin{split}
		{\Bc}^L_{\Lambda} g(n) & := \sum_{k=2}^n \binom{n}{k} \lambda_{n,k} (g(n-k+1)-g(n)), \ {\Bc}^L_{S} g(n) := \sum_{\ell=1}^n \binom{n}{\ell}\sigma_{n,\ell} (g(n+\ell)-g(n)),  \\
		{\Bc}^L_{D} g(n) & := \sum_{\ell=2}^{\kappa} n\beta_\ell (g(n+\ell-1)-g(n)),  \
		{\Bc}^L_{M, \nu} g(n) := \sum_{\ell=1}^n \binom{n}{\ell}m_{n,\ell} (g(n-\ell)-g(n)), \\
		{\Bc}^L_{M, \theta} g(n) &:=  n \theta (g(n-1)-g(n)). 
	\end{split}
\end{equation}
(Recall that $\sigma_{n,\ell}= \sigma_{n,\ell}^a+\sigma_{n,\ell}^A$ and $m_{n,\ell}= m_{n,\ell}^a+m_{n,\ell}^A$.)
We start this section with two useful lemmas about the processes $L$ and $V$. 
\begin{lemme}\label{non-exp}
	Let $g:\Nb_0\to\Rb$ be defined via $g(n)\coloneqq n^2$ for all $n \in \N_0$. There is a constant $C \in (0,\infty)$ such that ${\Bc}^Lg(n) \leq Cg(n)$ for all $n \in \N_0$. Moreover, the process $L$ is conservative. 
\end{lemme}
\begin{proof}
	Set ${\Bc}^L_{M} :=  {\Bc}^L_{M, \nu}+{\Bc}^L_{M, \theta}$.
	We have for all $n \in \N_0$, 
	\begin{align}
		{\Bc}^L_{\Lambda} g(n) & \leq 0, \qquad {\Bc}^L_{M} g(n) \leq 0,\\
		{\Bc}^L_{D} g(n) &= 2n^2 \sum_{\ell=2}^{\kappa} \beta_\ell (\ell-1) + n \sum_{\ell=2}^{\kappa} \beta_\ell (\ell-1)^2 \leq 3 g(n) \sum_{\ell=2}^{\kappa} \beta_\ell (\ell-1)^2, \nonumber \\
		{\Bc}^L_{S} g(n) & = \int_{(-1,1)} \left ( \sum_{\ell=0}^n (\ell^2+2n\ell) \binom{n}{\ell} |z|^{\ell}(1-|z|)^{n-\ell} \right )\frac{ \mu(\dd z)}{|z|} \nonumber \\
		& = \int_{(-1,1)} \left ( n|z|(1-|z|)+(n|z|)^2+2n^2|z| \right ) \frac{ \mu(\dd z)}{|z|} \leq 4 n^2 \mu((-1,1)) = 4 g(n) \mu((-1,1)), \label{contgenenv}
	\end{align}
	where we used the definitions of $\sigma_{n,\ell}^a$ and $\sigma_{n,\ell}^A$ (see \eqref{defrateposjumpenv}). Setting {$C:=4\mu((-1,1))+3\sum_{\ell=2}^{\kappa} \beta_\ell (\ell-1)^2$}, we get that ${\Bc}^Lg(n) \leq Cg(n)$ for all $n \in \N_0$. Note that the function $g$ is norm-like since $g(n)\to\infty$ as $n\to\infty$, see \cite{MeTwee} for the definition. Therefore $L$ is conservative by \cite[Thm. 1 (i)]{MeTwee}. 
\end{proof}

\begin{lemme}\label{controlrates}
	There is a constant $C \in (0,\infty)$ such that for any $n \in \N_0$ we have 
	\begin{align}
		\sum_{k=2}^n \binom{n}{k} \lambda_{n,k} + \sum_{\ell=1}^n \binom{n}{\ell}\sigma_{n,\ell} + \sum_{\ell=2}^{\kappa} n\beta_\ell + \sum_{\ell=1}^n \binom{n}{\ell}m_{n,\ell} + n \theta \leq Cn^2. \label{controlrates0}
	\end{align}
\end{lemme}
In the light of Definition~\ref{def:BCP}, Inequality \eqref{controlrates0} can be rephrased by saying that: for any $v \in \R^{\infty}$ the total transition rate of the process $V$ away from the state $v$ is smaller than $C (\dim(v)-1)^2$. 
\begin{proof}[Proof of Lemma~\ref{controlrates}]
	For $n>1$, using~\eqref{defrateneutjump} and that $(1-z)^{n-1}\geq 1-z(n-1)$, we get
	\begin{align}
		{\sum_{k=2}^n} \binom{n}{k} \lambda_{n,k} = & \binom{n}{2} \Lambda(\{0\}) + \int_{(0,1]} \left ( \sum_{k=2}^n \binom{n}{k} z^{k}(1-z)^{n-k} \right ) z^{-2} \Lambda(\dd z) \nonumber \\
		= & \frac{n(n-1)}{2} \Lambda(\{0\}) + \int_{(0,1]} \left ( 1- (1-z)^{n} - nz(1-z)^{n-1} \right ) z^{-2} \Lambda(\dd z) \nonumber \\
		= & \frac{n(n-1)}{2} \Lambda(\{0\}) + \int_{(0,1]} \left ( \sum_{k=0}^{n-1} ((1-z)^k-(1-z)^{n-1}) \right ) z^{-1} \Lambda(\dd z) \nonumber \\
		\leq & n^2 \Lambda(\{0\}) + n \int_{(0,1]} (1-(1-z)^{n-1}) z^{-1} \Lambda(\dd z)  \leq n^2 \Lambda([0,1]). \label{qudraticbound}
	\end{align}
	For $n \in \N$, using \eqref{defrateposjumpmut} and recalling that $m_{n,\ell}= m_{n,\ell}^a+m_{n,\ell}^A$, we compute
	\begin{align}
		\sum_{k=1}^n &\binom{n}{k} m_{n,k} =  \int_{(-1,1)} \left ( \sum_{k=1}^n \binom{n}{k}|z|^{k}(1-|z|)^{n-k} \right ) |z|^{-1} \nu(\dd z) \nonumber \\
		= & \int_{(-1,1)} \left ( 1- (1-|z|)^{n} \right ) |z|^{-1} \nu(\dd z) = \int_{(-1,1)} \left ( \sum_{k=0}^{n-1} (1-|z|)^k \right ) \nu(\dd z) \leq n \nu((-1,1)). \label{mutationbound}
	\end{align}
	A similar calculation, using~\eqref{defrateposjumpenv} and recalling that $\sigma_{n,\ell} = \sigma_{n,\ell}^a+\sigma_{n,\ell}^A$, shows  
	\begin{align}
		\sum_{\ell=1}^n \binom{n}{\ell}\sigma_{n,\ell} \leq n \mu((-1,1)). \label{envbound}
	\end{align}
	Combining \eqref{qudraticbound}, \eqref{mutationbound} and \eqref{envbound} yields \eqref{controlrates0} with $C\coloneqq \Lambda([0,1])+\mu((-1,1))+\nu((-1,1))+\theta+\sum_{\ell=2}^{\kappa}\beta_\ell$. 
\end{proof}
The proof of Theorem~\ref{thm:Bernsteinduality} follows the standard approach via generators. 
In other words, we prove that a duality holds between the generator of~$X$ and the rate matrix of the Bernstein coefficient process~$V$, and we then use \citep[Thm. 3.1(b)]{liggettduality} to deduce that the duality extends at the level of semigroups, yielding Theorem~\ref{thm:Bernsteinduality}. 
This requires to check a condition for the extension of the duality to hold. 
The idea behind this condition is that it ensures that the initial value problem satisfied by the semigroup of $V$ at the duality function has a unique solution, and that the rate matrix of $V$ can be approximated by bounded operators. 
We thus need to show that the rate matrix of $V$ satisfies the condition, which we describe here in the setting of~\citep{liggettduality}. 
Let $\mathcal{Y}$ be a countable set and $Q(\cdot,\cdot)$ be a rate matrix on $\mathcal{Y}$. 
\begin{cond} \label{conditionq}
	There is an increasing sequence $(\mathcal{Y}_n)_{n \in \N_0}$ of subsets of $\mathcal{Y}$ such that $\mathcal{Y}=\cup_{n \in \N_0} \mathcal{Y}_n$. There is a function $\varphi:\mathcal{Y}\rightarrow\R_+$ such that $\lim_{n\to\infty} \inf_{w \notin \mathcal{Y}_n}\varphi(w)=\infty$ and a constant $c>0$ such that: 
	\begin{align}
		\sup_{w \in \mathcal{Y}_n} \sum_{u \in \mathcal{Y}\setminus\{w\}} Q(w,u) & < \infty, n \in \N_0, \label{conditionq0} \\
		\sum_{u \in \mathcal{Y}} Q(w,u) (\varphi(u)-\varphi(w)) & \leq c \varphi(w), w \in \mathcal{Y}, \label{conditionq1} \\
		\sum_{u \in \mathcal{Y}\setminus\{w\}} Q(w,u) & \leq c \varphi(w), w \in \mathcal{Y}. \label{conditionq2}
	\end{align}
\end{cond}

The following lemma shows that the coefficient process $V$ satisfies Condition \ref{conditionq}, up to a restriction. 
\begin{lemme} \label{condq}
	For any countable subset $\mathcal{Y}$ of $\R^\infty$ that is invariant by $V$, the restriction of the process $V$ to $\mathcal{Y}$ satisfies Condition \ref{conditionq} with the sequence $(\mathcal{Y}_n)_{n \in \N}$ defined via $\mathcal{Y}_n\coloneqq\cup_{k\leq n} (\mathcal{Y}\cap\R^{k+1})$ and the function $\varphi: \mathcal{Y} \rightarrow \R_+$ given by $\varphi(w)\coloneqq(\dim(w)-1)^2$. As a consequence the process $V$ is conservative. 
\end{lemme}

\begin{proof}
	Let $(\mathcal{Y}_n)_{n \in \N}$ and $\varphi$ be given as in the statement of the lemma. Clearly $\mathcal{Y}=\cup_{n \in \N_0} \mathcal{Y}_n$ and $\inf_{w \notin \mathcal{Y}_n} \varphi(w)>n$. Recall also that $L_r=\dim(V_r)-1$ for all $r\geq 0$, and thus, for all $r\geq 0$, we have $\varphi(V_r)=g(L_r)$ with $g(\cdot)$ being as in Lemma \ref{non-exp}. Hence,  \eqref{conditionq1} follows from Lemma \ref{non-exp}. Moreover, Lemma \ref{controlrates} implies \eqref{conditionq2}. Finally, since $\mathcal{Y}_n=\{ w \in \mathcal{Y}, \varphi(w)\leq n^2 \}$, Condition \eqref{conditionq0} follows from \eqref{conditionq2}. 
	Noting that Condition \ref{conditionq} corresponds to Condition (Q) in \cite{liggettduality}, the fact that $V$ is conservative follows from the results in \cite{liggettduality}. 
\end{proof}

\subsection{Proof of the Bernstein duality}\label{sec:proofDuality}

In this section we prove Theorem~\ref{thm:Bernsteinduality} (Bernstein duality). We start with two lemmas that provide probabilistic arguments that will allow us to derive the duality relation at the level of the generator of $X$ and the rate matrix of $V$.

\begin{lemme}\label{computedistrib}
	Let $n \geq 0$ and $x,z \in [0,1]$. We set $(I(j))_{j \geq 0}$ and {$(R_{n,\ell,i}^a)_{\ell \in [n]_0, i \in [n+\ell]_0}$} to be collections of random variables such that $I(j)\sim\bindist{j}{x}$ for all $j\geq 0$ and $R_{n,\ell,i}^a$ is as in Section \ref{sec:asg}. We also set $J(z)$ to be a random variable such that $J(z)\sim\bindist{n}{z}$. We also assume that $(I(j))_{j \geq 0}$, {$(R_{n,\ell,i}^a)_{\ell \in [n]_0, i \in [n+\ell]_0}$} and $J(z)$ are mutually independent. Then we have 
	\begin{align}
		R_{n,J(z),I(n+J(z))}^a\sim\bindist{n}{x+zx(1-x)}. \label{computedistribexpr}
	\end{align}
\end{lemme}
\begin{proof}
	Let us construct a particular realization $Z$ of the random variable in the left-hand side of \eqref{computedistribexpr} and compute its distribution. Consider a set of $n$ \textit{old balls} numbered from $1$ to $n$. First, let each of these balls create two \textit{new balls} with probability $z$ or only one new ball with probability $1-z$. Second, color each of the new balls in red with probability $x$ and in blue with probability $1-x$. Third, {color old balls} according to the following rule: if an old ball created only one new ball, it gets the color of that ball; if an old ball created two new balls, it gets the color red if at least one of these balls is red {and the color blue otherwise}. We then define $Z$ to be the number of old balls colored in red. 
	
	On the one hand, note that $Z$ has the same distribution as the random variable in the left-hand side of \eqref{computedistribexpr}. On the other hand, the colors of the $n$ old balls are independent and each old ball is colored in red with probability $(1-z)x + z(1-(1-x)^2) = (1-z)x + z(2x-x^2) = x + zx(1-x).$
	Therefore $Z\sim\bindist{n}{x+zx(1-x)}$. Altogether, this proves \eqref{computedistribexpr}. 
\end{proof}

\begin{lemme}\label{computedistribmut}
	Let $n \geq 0$ and $x,z \in [0,1]$. We set $(\tilde I(j))_{j \geq 0}$ to be a collection of random variables such that $\tilde I(j)\sim\bindist{j}{x}$ for all $j\geq 0$, $J(z)$ to be a random variable such that $J(z)\sim\bindist{n}{z}$, and $(\tilde I(j))_{j \geq 0}$ and $J(z)$ to be independent. Then we have 
	\begin{equation}
		\tilde I(n-J(z))+J(z) \sim\bindist{n}{x+z(1-x)}. \label{computedistribmutexpr}
	\end{equation}
\end{lemme}
\begin{proof}
	Let us construct a realization $Z$ of the random variable in the left-hand side of \eqref{computedistribmutexpr} and compute its distribution. We consider a set of $n$ balls numbered from $1$ to $n$. We mark each of these balls with probability $z$, independently of each other. Each marked ball is then colored in red and each unmarked ball is colored in red with probability $x$ and in blue with probability $1-x$. We then define $Z$ to be the number of balls colored in red. On the one hand, one can see that $Z$ has the same distribution as $\tilde I(n-J(z))+J(z)$. On the other hand, the colors of the $n$ balls are independent, with {each ball} colored in red with probability $z+(1-z)x=x+z(1-x)$, so $Z \sim \bindist{n}{x+z(1-x)}$. We thus get \eqref{computedistribmutexpr}. 
\end{proof}

\begin{proof}[Proof of Theorem~\ref{thm:Bernsteinduality}]
	
	Let $(T_t)_{t\geq 0}$ be the semigroup of $X$, i.e. for any continuous function $f$ on $[0,1]$ and $x \in [0,1]$, $T_t f(x)\coloneqq \E_x[f(X_t)]$. Adapting the argument of \citep[Prop. A.4]{cordhumvech2022} to our case we can prove that $(T_t)_{t\geq 0}$ is a Feller semigroup with infinitesimal generator $\Ac$ acting on any twice continuously differentiable function $f$ on $[0,1]$ via 
	$\Ac f(x)=(\Ac_{\Lambda}+\Ac_{D}+\Ac_{S^a}+\Ac_{S^A}+\Ac_{M^a}+\Ac_{M^A})f(x),$
	where 
	\begin{align*}
		\Ac_{{\Lambda}}f(x)&\coloneqq \int\limits_{(0,1]}\!\!\!\! \big ( x[f(x+z(1-x))-f(x)] + (1-x)[f(x-zx)-f(x)] \big ) \frac{\Lambda(\dd z)}{z^2},\\  
		&\quad+\frac{\Lambda(\{0\})x(1-x)f''(x)}{2},\!\quad\quad\qquad\qquad\Ac_{D}f(x)\coloneqq d_{\sel}(x) f'(x),\\
		\Ac_{S^a}f(x)&\coloneqq\!\!\int\limits_{(0,1)}\!\! [ f(x+zx(1-x))-f(x) ] \frac{\mu(\dd z)}{z}, \quad\Ac_{S^A}f(x)\coloneqq\!\!\int\limits_{(-1,0)}\!\!\!\![ f(x+zx(1-x))-f(x) ] \frac{\mu(\dd z)}{|z|}, \nonumber \\
		\Ac_{M^a}f(x)&\coloneqq\!\!\int\limits_{(0,1)} [f(x+z(1-x))-f(x)]\frac{\nu(\dd z)}{z},\quad\Ac_{M^A}f(x)\coloneqq\!\!\int\limits_{(-1,0)} [f(x+zx)-f(x)]\frac{\nu(\dd z)}{|z|}.\nonumber 
	\end{align*}
	
	Let us fix $v\in \R^\infty$. We denote by  $(P_r)_{r\geq 0}$ and $\Bc$ respectively the semigroup and the rate matrix of $V$ on $C_V(v)$ (defined in \eqref{defcvv}), where $\Bc$ is defined in a similar way as $\Bc^L$, see Section \ref{secfirstprop}. By Definition \ref{def:BCP}, $\Bc(w,\cdot)$ has finite support for all $w \in C_V(v)$. For any function $g:C_V(v)\to\Rb$ we define ${\Bc}g$ 
	via ${\Bc} g(w)=\sum_{u \in C_V(v)} {\Bc}(w,u) (g(u)-g(w))$. By Definition \ref{def:BCP} $\Bc$ acts on functions $g:C_V(v)\to\Rb$ via 
	${\Bc}g(w)=({\Bc}_{D}+{\Bc}_{\Lambda}+{\Bc}_{S^a}+{\Bc}_{S^A}+{\Bc}_{M^a}+{\Bc}_{M^A})g(w),$ where for $w\in C_V(v)\cap\R^{n+1}$,
	\begin{align*}
		{\Bc}_{D}g(w)&:= n\sum_{\ell=1}^{\kappa}\beta_\ell[g(D^{n,\ell}w)-g(w)], & & {\Bc}_{\Lambda}g(w):= \sum_{k=2}^n \binom{n}{k} \lambda_{n,k} [g({C}^{n,k}w)-g(w)], \nonumber \\
		{\Bc}_{S^a}g(w)&:= \sum_{\ell=1}^n \binom{n}{\ell}\sigma^a_{n,\ell}[g(S_a^{n,\ell}w)-g(w)],  && {\Bc}_{S^A}g(w):= \sum_{\ell=1}^n \binom{n}{\ell}\sigma^A_{n,\ell}[g(S_A^{n,\ell}w)-g(w)], \\
		{\Bc}_{M^a}g(w)&:= \sum_{\ell=1}^n \binom{n}{\ell}m^a_{n,\ell}[g(M_a^{n,\ell}w)-g(w)],  && {\Bc}_{M^A}g(w):= \sum_{\ell=1}^n \binom{n}{\ell}m^A_{n,\ell}[g(M_A^{n,\ell}w)-g(w)].
	\end{align*}
	{Recall the definition of $H(\cdot,\cdot)$ in \eqref{defh}}. The target duality from Theorem \ref{thm:Bernsteinduality} can then be re-written as $T_t H(\cdot,v)(x)= P_t H(x,\cdot)(v)$. Using Lemma \ref{lem:actionV}, we obtain
	\begin{equation}\label{int1}
		\sup_{x\in[0,1],w\in C_V(v)} \lvert H(x,w)\rvert \leq \lVert v\rVert_\infty, 
	\end{equation}
	so $H(\cdot,\cdot)$ is bounded on $[0,1]\times C_V(v)$. For any $w\in C_V(v)$, $x\mapsto H(x,w)$ is infinitely differentiable, and therefore the function belongs to the domain of $\Ac$. By Lemma~\ref{condq} the process $V$ on $C_V(v)$ satisfies Condition~\ref{conditionq} (and thus Condition~(Q) in \cite{liggettduality}).
	According to \citep[Thm. 3.1(b)]{liggettduality} the duality $T_t H(\cdot,v)(x)= P_t H(x,\cdot)(v)$ will follow if we can prove the duality for the operators $\Ac$ and $\Bc$. More precisely, it is sufficient to prove that for $x\in [0,1]$ and $w\in \R^{\infty}$, we have \begin{equation}
		\Ac H(\cdot,w)(x)=\Bc H(x,\cdot)(w).\label{eq:generatorrelation}
	\end{equation}
	
	It was already shown in~\citep[Proof Thm. 2.14]{Cordero2022} that $\Ac_{\Lambda} H(\cdot,w)(x)={\Bc}_{\Lambda} H(x,\cdot)(w)$. Similarly, it is not difficult to modify~\citep[Proof Thm. 2.14]{Cordero2022} to show that $\Ac_{D} H(\cdot,w)(x)={\Bc}_{D} H(x,\cdot)(w)$. 
	
	Next, we show $\Ac_{S^a}H(\cdot,w)(x)= {\Bc}_{S^a}H(x,\cdot)(w)$. Let us fix $x \in [0,1]$, $n\in \N$, and $w\in \R^{n+1}$. Using the definitions of $\Bc_{S^a}$, $\sigma^a_{n,\ell}$ (see \eqref{defrateposjumpenv}), $H(\cdot,\cdot)$, and  $S_a^{n,\ell}$ (see Def. \ref{def:bramutoperators}), and using that $S_a^{n,0}w=w$, we get 
	\begin{align*}
		\Bc_{S^a}& H(x,\cdot)(w)  = \sum_{\ell=1}^n \binom{n}{\ell}\sigma^a_{n,\ell}[H(x,S_a^{n,\ell}w)-H(x,w)] \\
		%& = \int_{(0,1)} \left ( \sum_{\ell=0}^n \binom{n}{\ell}z^{\ell-1}(1-z)^{n-\ell}[H(x,S_a^{n,\ell}w)-H(x,w)] \right ) \mu(\dd z) \\
		& = \int_{(0,1)} \left ( \sum_{\ell=0}^n \binom{n}{\ell}z^{\ell-1}(1-z)^{n-\ell} \left [\sum_{i=0}^{n+\ell} (S_a^{n,\ell}w)_i \binom{n+\ell}{i}x^i(1-x)^{n+\ell-i}-H(x,w) \right ] \right ) \mu(\dd z) \\
		& = \int_{(0,1)} \left ( \sum_{\ell=0}^n \binom{n}{\ell}z^{\ell-1}(1-z)^{n-\ell} \left [\sum_{i=0}^{n+\ell} \E\left[w_{R_{n,\ell,i}^{a}}\right] \binom{n+\ell}{i}x^i(1-x)^{n+\ell-i}-H(x,w) \right ] \right ) \mu(\dd z) \\
		& = \int_{(0,1)} \left ( \E \left [ w_{R_{n,J(z),I(n+J(z))}^a} \right ] -H(x,w) \right ) z^{-1} \mu(\dd z), 
	\end{align*}
	where the random variable $R_{n,J(z),I(n+J(z))}^a$ in the last equality is the same as in Lemma \ref{computedistrib}. Using that lemma we get that this random variable has law $\bindist{n}{x+zx(1-x)}$ so the above equals 
	\begin{align*}
		& \int_{(0,1)} \left ( \sum_{i=0}^{n} w_i \binom{n}{i}(x+zx(1-x))^i(1-x-zx(1-x))^{n-i} -H(x,w) \right ) z^{-1} \mu(\dd z) \\
		= & \int_{(0,1)} \left ( H(x+zx(1-x),w) -H(x,w) \right ) z^{-1} \mu(\dd z) = \Ac_{S^a} H(\cdot,w)(x), 
	\end{align*}
	where we used the definitions of $H(\cdot,\cdot)$ and $\Ac_{S^a}$. A similar calculation yields $\Ac_{S^A} H(\cdot,w)(x)={\Bc}_{S^A} H(\cdot,w)(x)$. 
	
	Next, we show $\Ac_{M^a}H(\cdot,w)(x)= {\Bc}_{M^a}H(x,\cdot)(w)$. Fix $x \in [0,1]$, $n\in \N$, and $w\in \R^{n+1}$. Using the definitions of $\Bc_{M^a}$, $m^a_{n,\ell}$ (see \eqref{defrateposjumpmut}), $H(\cdot,\cdot)$, and $M_a^{n,\ell}$ (see Def. \ref{def:bramutoperators}), and using that $M_a^{n,0}w=w$, we get 
	\begin{align*}
		\Bc_{M^a} & H(x,\cdot)(w)  = \sum_{\ell=1}^n \binom{n}{\ell}m^a_{n,\ell}[H(x,M_a^{n,\ell}w)-H(x,w)] \\
		%& = \int_{(0,1)} \left ( \sum_{\ell=0}^n \binom{n}{\ell}z^{\ell-1}(1-z)^{n-\ell}[H(x,M_a^{n,\ell}w)-H(x,w)] \right ) \nu(\dd z) \\
		& = \int_{(0,1)} \left ( \sum_{\ell=0}^n \binom{n}{\ell}z^{\ell-1}(1-z)^{n-\ell} \left [\sum_{i=0}^{n-\ell} (M_a^{n,\ell}w)_i \binom{n-\ell}{i}x^i(1-x)^{n-\ell-i}-H(x,w) \right ] \right ) \nu(\dd z) \\
		& = \int_{(0,1)} \left ( \sum_{\ell=0}^n \binom{n}{\ell}z^{\ell-1}(1-z)^{n-\ell} \left [\sum_{i=0}^{n-\ell} w_{i+\ell} \binom{n-\ell}{i}x^i(1-x)^{n-\ell-i}-H(x,w) \right ] \right ) \nu(\dd z) \\
		& = \int_{(0,1)} \left ( \E \left [ w_{\tilde I(n-J(z))+J(z)} \right ] -H(x,w) \right ) z^{-1} \nu(\dd z), 
	\end{align*}
	where the random variable $\tilde I(n-J(z))+J(z)$ in the last equality is the same as in Lemma \ref{computedistribmut}. Using that lemma we get that this random variable has law $\bindist{n}{x+z(1-x)}$. In particular, the above equals
	\begin{align*}
		& \int_{(0,1)} \left ( \sum_{i=0}^{n} w_i \binom{n}{i}(x+z(1-x))^i(1-x-z(1-x))^{n-i} -H(x,w) \right ) z^{-1} \nu(\dd z) \\
		= & \int_{(0,1)} \left ( H(x+z(1-x),w) -H(x,w) \right ) z^{-1} \nu(\dd z) = \Ac_{M^a} H(\cdot,w)(x), 
	\end{align*}
	where we have used the definitions of $H(\cdot,\cdot)$ and $\Ac_{M^a}$. A similar calculation leads to $\Ac_{M^A} H(\cdot,w)(x)={\Bc}_{M^A} H(\cdot,w)(x)$. 
	Altogether, this proves \eqref{eq:generatorrelation}, which completes the proof. 
\end{proof}

\subsection{\texorpdfstring{Useful properties of the line-counting process $L$}{Useful properties of the line-counting process L}}\label{sec:usefulpropertiesofL}

As mentioned in the introduction, for applications that follow, we require the process $L$ to be positive recurrent if there is no mutation, and to be almost surely absorbed if there is mutation. We already alluded in the main result section to the fact that~\eqref{Assump} provides the appropriate condition. Let us recall this. To this end, define
$$b(\beta)\coloneqq \sum_{\ell=2}^\kappa \beta_\ell (\ell-1),\quad\text{and}\quad c(\Lambda)\coloneqq \int_{[0,1]}|\log(1-z)|\frac{\Lambda(\dd z)}{z^2},$$
and recall that $\theta=\theta_a+\theta_A$. Then~\eqref{Assump} reads
$b(\beta)+\mu(-1,1)<c(\Lambda)+\nu(-1,1)+\theta.$
In the remainder of this section, we will show, using ideas developed in~\cite{foucart2013impact}, that under \eqref{Assump}, the process $L$ is either positive recurrent or almost surely absorbed in finite time. Define $\delta:\N\to\R$ and $f:\N \to \R$ via
\begin{align*}
	\delta(n)&\coloneqq \binom{n}{2}\Lambda(\{0\})-n\int_{(0,1]}\log\left(1-\frac{1}{n}(nz-1+(1-z)^n)\right)\frac{\Lambda(\dd z)}{z^2},\\
	f(\ell)&\coloneqq\sum_{k=2}^\ell \frac{k}{\delta(k)}\log\Big( \frac{k}{k-1}\Big).
\end{align*}
The next lemma is a generalization of \cite[Lemma 2.3]{foucart2013impact}, which corresponds to the case $\beta_2>0$, $\beta_\ell=0$ for $\ell\neq 2$, and $p_{1}^{(2)}=0$.
\begin{lemme}\label{lem:generatorinequality}
	Recall that ${\Bc}^L$, defined in Section \ref{secfirstprop}, is the rate matrix of the process $L$. Assume that \eqref{Assump} is satisfied. Then there exists $n_0\in \N$ and $\varepsilon>0$ such that for all $n\geq n_0$ $${{\Bc}^L f(n)}\leq -1+ \frac{b(\beta)+\mu(-1,1)-(\nu(-1,1)+\theta)}{c(\Lambda)}+\varepsilon \left(b(\beta)+\mu(-1,1)\right)<0,$$
	with the usual convention that $1/\infty=0$.
\end{lemme}
\begin{proof}
	Recall that ${\Bc}^L=({\Bc}^L_{\Lambda}+{\Bc}^L_{S}+{\Bc}^L_{D}+{\Bc}^L_{M, \nu}+{\Bc}^L_{M, \theta})$ with the building blocks defined in \eqref{genl}. In the proof of \cite[Lemma 5.3]{Cordero2022} it was shown that, for any $\varepsilon>0$ there is $n_0\in\Nb$ such that for all $n\geq n_0$ 
	\begin{equation}\label{lpre}
		\frac{n}{\delta(n)}\leq \frac{1}{c(\Lambda)}+\varepsilon\quad\text{and}\quad({{\Bc}^L_{\Lambda}+{\Bc}^L_{D}})f(n)\leq -1+\frac{b(\beta)}{c(\Lambda)}+\varepsilon b(\beta).
	\end{equation}
	Moreover, the function $n\mapsto n/\delta(n)$ is non-increasing and converges to $1/c(\Lambda)$. Thus, for all $n\in\Nb$,
	\begin{equation}\label{lb}
		\frac{1}{c(\Lambda)}\leq \frac{n}{\delta(n)}. 
	\end{equation}
	Let us now see the effect of the other operators. First note that, using \eqref{lb}, we obtain
	\begin{align}\label{uth}
		{{\Bc}^L_{M, \theta} f(n)}&=n\theta(f(n-1)-f(n))=-n\theta\frac{n}{\delta(n)}\log\left(\frac{n}{n-1}\right)\leq n\frac{\theta}{c(\Lambda)}\log\left(1-\frac{1}{n}\right)\leq -\frac{\theta}{c(\Lambda)}. 
	\end{align}
	Similarly, for $k\in [n]$,
	\begin{align*}
		f(n-k)-f(n)&\leq-\frac{1}{c(\Lambda)}\sum_{j=n-k+1}^n\log\left(\frac{j}{j-1}\right){=} \frac{1}{c(\Lambda)}\log\left(1-\frac{k}{n}\right)\leq-\frac{1}{c(\Lambda)}\frac{k}{n}.
	\end{align*}
	Therefore, using that ${\tiny \binom{n}{k}}\frac{k}{n}={\tiny \binom{n-1}{k-1}}$
	\begin{align}\label{uM}
		{{\Bc}^L_{M, \nu} f(n)}&\leq- \frac{1}{c(\Lambda)}\int\limits_{(-1,1)} \left ( \sum_{k=1}^n \binom{n-1}{k-1}|z|^{k-1}(1-|z|)^{n-k} \right ) \nu(\dd z)=-\frac{\nu(-1,1)}{c(\Lambda)}. 
	\end{align}
	Finally, let us consider ${{\Bc}^L_{S}}$. To this end, note that, thanks to \eqref{lpre}, for $n\geq n_0$
	\begin{align*}
		f(n+\ell)-f(n)&\leq\left(\frac{1}{c(\Lambda)}+\varepsilon\right)\sum_{j=n+1}^{n+\ell}\log\left(\frac{j}{j-1}\right){=} \left(\frac{1}{c(\Lambda)}+\varepsilon\right)\log\left(1+\frac{\ell}{n}\right)\leq\left(\frac{1}{c(\Lambda)}+\varepsilon\right)\frac{\ell}{n}.
	\end{align*}
	Hence, proceeding as for ${{\Bc}^L_{M, \nu} f(n)}$, we get
	\begin{align}\label{uS}
		{{\Bc}^L_{S} f(n)}&\leq \left(\frac{1}{c(\Lambda)}+\varepsilon\right)\mu(-1,1).
	\end{align}
	The first inequality in the statement follows combining \eqref{lpre}, \eqref{uth}, \eqref{uM} and \eqref{uS}. Since under \eqref{Assump}, 
	$$-1+ \frac{b(\beta)+\mu(-1,1)-(\nu(-1,1)+\theta)}{c(\Lambda)}<0,$$
	the second inequality follows by choosing first $\varepsilon>0$ sufficiently small so that the middle term is strictly smaller than $0$ and then choosing $n_0$ for that choice of $\varepsilon.$  
\end{proof}

\begin{lemme}[Recurrence and absorption of~$L$]\label{lc-pos_abs}
	Assuming that~\eqref{Assump} is satisfied, the following assertions hold.
	\begin{enumerate}
		\item If \eqref{eq:nomut} holds, then $1$ is positive recurrent for $L$. 
		\item If \eqref{eq:nomut} does not hold,
		then $L$ is absorbed at $0$ in finite time. 
	\end{enumerate}
\end{lemme}
\begin{proof}
	Let $T^{(n)}\coloneqq \inf\{r\geq 0:L_r<n\}$, $n\in\Nb.$
	Since the process~$L$ is conservative by Lemma \ref{non-exp} and $f$ is bounded, $(f,\Bc^Lf)$ is in the full generator. 
	Thus, we can use Lemma \ref{lem:generatorinequality} and proceed analogously to the proof of \cite[Lemma 5.4]{Cordero2022} (see also \cite[Lemma 2.4]{foucart2013impact}), to show that there is $n_0\in\Nb$ and a constant $c>0$ such that
	\begin{equation}\label{recu}
		\Eb_n[T^{(n_0)}]<cf(n),\quad n\geq n_0.
	\end{equation}
	Note that one can construct $p\in(0,1)$ such that every time the process $L$ enters $[n_0]_0$ the probability that $L$ hits {$0/1$} before leaving $[n_0]_0$ is at least $p$. In the case without mutation, {we choose $p$} as the probability of going from $n_0$ to $1$ by consecutive coalescences; in the case with mutation {we choose $p$} as the probability of going from $n_0$ to $0$ by consecutive mutations. 
	{Considering the possible transitions of the line-counting process $L$, we see that, directly upon leaving $[n_0]_0$, $L$ arrives in the finite set $\{n_0+1,\ldots,n_1\}$ where $n_1:=2n_0\vee (n_0+\kappa)$. This and \eqref{recu} show that, every time $L$ leaves $[n_0]_0$, the return time to $[n_0]_0$ has finite expectation. Combining both arguments, we get that in the case without (resp. with) mutation the expectation of the hitting time of $1$ (resp. $0$) by $L$ is finite under~$\P_n$ for any $n\in\N$.} Thus, $1$ is positive recurrent for $L$ in case (1) and $L$ is absorbed at $0$ in finite time in case (2).
\end{proof}

\section{Proof of other main results}\label{sec:proofmainres}

\subsection{Properties of the Bernstein dual: Proof of Propositions \ref{prop:asympV} and \ref{prop:absV}}\label{sec:propB}

We now use the results from Section \ref{sec:lcp} to prove the properties of the Bernstein coefficient process stated in Section \ref{S2}. 
We start with the following lemma, which provides some basic properties of the operators introduced in Definition~\ref{def:bramutoperators}.
\begin{lemme}\label{lem:actionV}
	Let $n\in \N$ and $v\in \R^{n+1}$. Then for any $k\in (n]$, $l\in(\kappa]$ and $\ell\in[n]$,
	\begin{align*}
		&(C^{n,k}v)_0 =(D^{n,l}v)_0 =(S_a^{n,\ell}v)_0=(S_A^{n,\ell}v)_0= v_0, \\
		&(C^{n,k}v)_{n-k+1}=(D^{n,l}v)_{n+l-1}= (S_a^{n,\ell}v)_{n+\ell}=(S_A^{n,\ell}v)_{n+\ell}=v_n, \\
		&\lVert C^{n,k} v\rVert_{\infty},\ \lVert D^{n,l} v\rVert_{\infty},\ \lVert S^{n,\ell}_a v\rVert_{\infty},\ \lVert S^{n,\ell}_A v\rVert_{\infty},\ \lVert M^{n,\ell}_a v\rVert_{\infty}, \ \lVert M^{n,\ell}_A v\rVert_{\infty}\ \leq \lVert v\rVert_{\infty}.
	\end{align*}	
\end{lemme}
\begin{proof}
	The result follows in a straightforward way from the definition of the operators.
\end{proof}

\begin{proof}[Proof of Propositions \ref{prop:asympV} and \ref{prop:absV}] 
	Proposition \ref{prop:absV} follows directly from Lemma \ref{lc-pos_abs}. 
	We now assume \eqref{eq:nomut} and we prove Proposition~\ref{prop:asympV}. 
	Lemma~\ref{lem:actionV} yields directly that the functions $r\mapsto V_{r}(0)$ and $r\mapsto V_{r}(L_r)$ are constant. 
	In particular, the restriction of $V$ to the set $C^0_V((a,b)^T)$ (defined in \eqref{defcvv}) is an irreducible continuous-time Markov chain. Indeed, for $w,w'\in C^0_V((a,b)^T)$ to go from $w$ to $w'$, first go from $w$ to $(a, b)^T$ by successive coalescence operations. Then go from $(a, b)^T$ to $w'$ in a finite number of successive selection and coalescence operations.
	Let us assume now that \eqref{Assump} is satisfied. 
	\smallskip
	
	(1) By Lemma~\ref{lc-pos_abs}, \eqref{Assump} implies that $1$ is positive recurrent for $L$. Since, in addition, the functions $r\mapsto V_{r}(0)$ and $r\mapsto V_{r}(L_r)$ are constant, the state $(a,b)^T$ is positive recurrent for $V$. By irreducibility, the process $V$ on $C^0_V((a,b)^T)$ is positive recurrent, and hence it admits a unique invariant distribution $\mu^{a,b}$~\cite[Thm. 3.5.2, Thm. 3.5.3]{norris1998markov} with support $C^0_V((a,b)^T)$. 
	
	\smallskip
	
	(2) Let $V_{\infty}^{a,b}\sim \mu^{a,b}$. If $V_0=(a,b)^T$, then $V_r\to V_{\infty}^{a,b}$ in law as $r\to\infty$ by classic Markov chain theory~\citep[Thm. 3.6.2]{norris1998markov}. On the other hand, for $v \in \Rb^{\infty}$ with $v_0=a$ and $v_{\dim(v)-1}=b$, since $1$ is positive recurrent for $L$, the process $V$ starting at $V_0=v$ enters $C^0_V((a,b)^T)$ in finite time $\P_{v}$-almost surely (in particular $C^0_V((a,b)^T) \subset C^0_V(v)$). Hence, the result of Proposition \ref{prop:asympV} follows. 
\end{proof}

\subsection{The case without mutation: Proof of Theorem \ref{fixext}} \label{sec:prooffixation}
In this section we assume $\vec{\theta}=0$ and $\nu=0$ and study fixation and extinction. 
\begin{lemme} \label{lemf}
	Assume that \eqref{Assump} and \eqref{eq:nomut} are satisfied. Let $h:[0,1]\rightarrow [0,\infty)$ be defined by $h(x)\coloneqq \E[H(x,V^{0,1}_{\infty})]$, where the r.v. $V^{0,1}_{\infty}$ is as in Proposition \ref{prop:asympV}. For any $x \in [0,1]$ and $k \geq 1$, we have 
	\begin{align}
		h(x)=\lim_{r\to\infty} \E_{e_k}[ H(x,V_r)], \label{feqlimhxvt}
	\end{align}
	where the limit in the right-hand side does not depend on the choice of $k \geq 1$. 
	Moreover $h$ is continuous and satisfies $h(0)=0$, $h(1)=1$ and $h(x) \in (0,1)$ for $x \in (0,1)$. \end{lemme}

\begin{proof}
	By Proposition~\ref{prop:asympV}, if $V_0=e_k$, $V_r$ converges in distribution to $V^{0,1}_{\infty}$. By \eqref{int1} we have that, for any $x,v \in [0,1] \times C^0_V(e_k)$, $|H(x,v)|\leq 1$. By the continuous mapping theorem we get \eqref{feqlimhxvt}. Since $V^{0,1}_{\infty}$ is supported on $C^0_V(e_1)$ by Proposition~\ref{prop:asympV}, and since the mapping $x \mapsto H(x,v)$ is continuous for any $v \in C^0_V(e_1)$, and $\lvert H(x,v)\rvert \leq 1$ for any $x,v \in [0,1] \times C^0_V(e_1)$, we get by dominated convergence that $h$ is continuous. It is plain that for any $v \in C^0_V(e_1)$, $H(0,v)=0$ and $H(1,v)=1$, so $h(0)=0$ and $h(1)=1$. We now justify that $h(x) \in (0,1)$ for $x \in (0,1)$. Note that for $v=e_1$ we have $H(x,v)=x$. Moreover, by Proposition \ref{prop:asympV} we see that $e_1$ is in the support of $\mu^{0,1}$, i.e. $\P(V^{0,1}_{\infty}=e_1)>0$. We have 
	\begin{align}
		h(x) = x \, \P(V^{0,1}_{\infty}=e_1) + \sum_{v \in C^0_V(e_1) \setminus \{e_1\}} H(x,v) \, \P(V^{0,1}_{\infty}=v). \label{combcvx}
	\end{align}
	For any $x,v \in [0,1] \times C^0_V(e_1)$ we have $H(x,v)\in [0,1]$ (because $|H(x,v)|\leq 1$ and $H(x,v)\geq0$). Since $\sum_{v \in C^0_V(e_1)} \P(V^{0,1}_{\infty}=v)=1$ and $\P(V^{0,1}_{\infty}=e_1)>0$, we get from \eqref{combcvx} that $h(x) \in (0,1)$ for $x \in (0,1)$. 
\end{proof}

\begin{proof}[Proof of Theorem \ref{fixext}]
	Let $h$ be as in Lemma \ref{lemf}. Let us show that $(h(X_t))_{t\geq 0}$ is a bounded martingale. To this end, we use the Markov property for $X$, Lemma \ref{lemf}, dominated convergence (which can be used since, by \eqref{int1}, we have $|\E_{e_1}[H(\tilde X_s,V_r)]|\leq \lVert e_1\rVert_\infty=1$), Fubini's theorem, {Theorem \ref{thm:Bernsteinduality}}, and the Markov property for $V$, to obtain 
	\begin{align*}
		\E_x&[h(X_{t+s})\mid \mathcal{F}_t]= \E_{X_t}[h(\tilde X_{s})]=\E_{X_t}[ \lim_{r\to\infty} \E_{e_1}[H(\tilde X_s,V_r)]]=\lim_{r\to\infty} \E_{X_t}[ \E_{e_1}[H(\tilde X_s,V_r)]] \\
		&= \lim_{r\to\infty}\E_{e_1}[\E_{X_t}[H(\tilde X_s,V_r)]]= \lim_{r\to\infty}\E_{e_1}[\E_{V_r}[H(X_t,\tilde V_s)]]=\lim_{r\to\infty} \E_{e_1}[H(X_t,\tilde V_{r+s})]=h(X_t).
	\end{align*}
	Now, by Doob's martingale convergence theorem, $F:=\lim_{t\to\infty} h(X_t)$ exists almost surely. Using {Theorem \ref{thm:Bernsteinduality}} and Lemma \ref{lemf} we get $\E_x[X_t^k]=\E_x[H(X_t,e_k)]=\E_{e_k}[ H(x,V_t)] \to h(x)$ as $t\to\infty$. 
	
	Therefore the positive integer moments of $X_t$ converge to the moments of the distribution $\berno{h(x)}$. Probability distributions on $[0,1]$ are completely determined by their positive integer moments, and convergence of those moments imply convergence in distribution. We thus get that the law of $X_t$ converges weakly to $\berno{h(x)}$. Since, by Lemma \ref{lemf}, $h$ is continuous and satisfies $h(0)=0$ and $h(1)=1$, we get by the continuous mapping theorem that the law of $h(X_t)$ converges weakly to $\berno{h(x)}$. Combining with the above almost sure convergence of $h(X_t)$, we get $F \sim \berno{h(x)}$; so in particular $F \in \{0,1\}$ a.s. Let us show that $X_t$ converges to $F$ almost surely. Let $\epsilon \in (0,1/2)\cap \mathbb{Q}$. Since, by Lemma \ref{lemf}, $h$ is continuous, positive on $(0,1]$, and strictly smaller than $1$ on $[0,1)$, we have $\delta(\epsilon):=\min \{ \min_{z \in [\epsilon,1]}h(z),1-\max_{z \in [0,1-\epsilon]}h(z) \}>0$. Since $h(X_t)$ converges to $F$ almost surely, there exists almost surely $T(\epsilon)$ such that $|h(X_t)-F|<\delta(\epsilon)$ for all $t \geq T(\epsilon)$. Distinguishing the cases $F=0$ and $F=1$ we see that this is possible only if $|X_t-F|<\epsilon$ for all $t \geq T(\epsilon)$. In conclusion, $\P(|X_t-F|<\epsilon \ \text{for all large $t$})=1$. Since this is true for all $\epsilon \in (0,1/2)\cap \mathbb{Q}$ we get that $X_t$ converges almost surely to $F$. This proves that the first claim holds with $X_\infty=F$. Finally, \eqref{exprfivproba} follows from $F \sim \berno{h(x)}$ and Lemma \ref{lemf}. 
\end{proof}

\subsection{The case with mutation: Proof of Theorem~\ref{thstadist} and Proposition \ref{recursionmoments}}\label{sec:proofstationary}
In this section, we work in the presence of bidirectional mutation, i.e. we assume \eqref{eq:bimut} and study the ergodicity of $X$ and its stationary distribution. 

\begin{proof}[Proof of Theorem \ref{thstadist}]
	Let $\tau \coloneqq \inf \{ r \geq 0:\, L_r=0 \}$. According to Proposition \ref{prop:absV}, for any $v \in \R^\infty$, we have $\P_{v}$-a.s. $\tau<\infty$ and, for $r \geq \tau$, $\dim(V_r)=1$ and $V_r=V_{\tau}=V_\infty$. In particular, for any $x \in [0,1]$, $H(x,V_r)=H(x,V_{\tau})=U_\infty$ for $r \geq \tau$. By Theorem \ref{thm:Bernsteinduality} we have $\E_x[X_t^k] = \E_x[H(X_t,e_k)] = \E_{e_k}[ H(x,V_t)]$. Since in addition,
	by \eqref{int1}, $\P_{e_k}$-a.s. we have $|H(x,V_r)|\leq 1$ for all $r\geq 0$, we get
	\begin{align}
		\left | \E_x[X_t^k] - \E_{e_k}[ U_\infty] \right | \leq \left | \E_{e_k}[ H(x,V_t) - H(x,V_{\tau}) ] \right | \leq 2 \P_{e_k} (\tau > t) \xrightarrow[t \rightarrow \infty]{} 0, \label{cvmomentsmut}
	\end{align}
	Therefore, all positive integer moments of $X_t$ converge to a limit (which does not depend on the initial condition $x$) as $t\to\infty$. Arguing as in the proof of Theorem \ref{fixext}, we find that there is a probability distribution $\mathcal{L}$ on $[0,1]$ such that for any $x\in[0,1]$, the law of $X_t$ under $\P_x$ converges in distribution to $\mathcal{L}$ as $t\to\infty$ and the moments of this law are given by \eqref{momentsloiinvar}. 
	By dominated convergence, the convergence of the law of $X_t$ towards $\mathcal{L}$ as $t\to\infty$ extends to any distribution on $[0,1]$ for $X_0$. Hence, $X$ has $\mathcal{L}$ as its unique stationary distribution of $X$. We see from Definition \ref{def:BCP} that, under \eqref{eq:bimut}, the event where the first transition of $V$ corresponds to a type $a$ (resp. $A$) mutation has positive $\P_{e_1}$-probability. Thus, $\P_{e_1}(U_\infty=1)>0$ and $\P_{e_1}(U_\infty=0)>0$.
	Since $\P_{e_1}$-a.s. all components of $V_r$ are non-negative for all $r$, we get by \eqref{int1} that $U_\infty \in [0,1]$ a.s., and hence $\rho_1 \in (0,1)$ (see \eqref{momentsloiinvar} for the definition of $\rho_1$). Let $\epsilon \in (0,1)$ and note that, by \eqref{cvmomentsmut}, there is $t_0>0$ such that $\E_x[X_{t_0}]<(1+\epsilon)^{1/2} \rho_1$ for all $x \in [0,1]$. In particular, by Markov inequality, $\P_x(X_{t_0}>(1+\epsilon) \rho_1)<(1+\epsilon)^{-1/2}$ for all $x \in [0,1]$. Therefore, by the Markov property, the set $[0,(1+\epsilon) \rho_1]$ is recurrent for the skeleton chain $(X_{nt_0})_{n \geq 0}$ and thus for $X$. A similar reasoning for $1-X$ shows that $[(1-\epsilon) \rho_1,1]$ is recurrent for $X$. In particular, we get $\P(\exists z \in [0,(1-\epsilon) \rho_1) \cup ((1+\epsilon) \rho_1,1] \ \text{s.t.} \ X_t \to z \textrm{ as $t\to\infty$})=0$. This concludes the proof. 
\end{proof}

\begin{remark}
	In \eqref{momentsloiinvar}, the expectation $\E_{v}[U_\infty]$ is expressed in terms of the sequence $(\rho_k)_{k \geq 1}$ when $v=e_n$ for some $n \in \N_0$. In order to prove Proposition \ref{recursionmoments}, we need such an expression for all $v \in \R^\infty$. Letting $t\to\infty$ in Theorem \ref{thm:Bernsteinduality} with $v\in\Rb^n$ and using Theorem \ref{thstadist}, Proposition \ref{prop:absV} and dominated convergence we get 
	\begin{align}
		\E_{v}[U_\infty] & = \sum_{i=0}^{n} \binom{n}{i} v_i \int_{[0,1]} z^i(1-z)^{n-i} \mathcal{L}(\dd z) = \sum_{i=0}^{n} \sum_{j=0}^{n-i} \binom{n}{i} \binom{n-i}{j} (-1)^j v_i \rho_{i+j} \nonumber \\
		& = \sum_{i=0}^{n} \sum_{k=i}^{n} \binom{n}{i} \binom{n-i}{k-i} (-1)^{k-i} v_i \rho_k = \sum_{k=0}^{n} \left ( \sum_{i=0}^{k} \binom{n}{i} \binom{n-i}{k-i} (-1)^{k-i} v_i \right ) \rho_k. \label{eveqmom}
	\end{align}
\end{remark}

\begin{proof}[Proof of Proposition \ref{recursionmoments}]
	Recall the transitions of the process $V$ given in Definition \ref{def:BCP} and the definition of $\alpha_n$ in the statement of the proposition. Let $n\geq 1$. A fist step decomposition of $V$ started at $V_0=e_n$ yields
	\begin{align}
		\alpha_n& \, \E_{e_n}[U_\infty]  = \sum_{\ell=2}^{n} {\binom{n}{\ell}} \lambda_{n,\ell} \E_{C^{n,\ell} e_n}[U_\infty] + n \sum_{\ell=2}^{\kappa}\beta_\ell \E_{D^{n,\ell} e_n}[U_\infty] \nonumber \\
		& +\sum_{c\in\{a,A\}}\bigg( \sum_{\ell=1}^{n} \binom{n}{\ell} \sigma_{n,\ell}^c \E_{S^{n,\ell}_c e_n}[U_\infty] + \sum_{\ell=1}^{n} \binom{n}{\ell} \left ( m^c_{n,\ell} + \theta_c \mathds{1}_{\ell=1} \right ) \E_{M^{n,\ell}_c e_n}[U_\infty]\bigg). \label{1stepdecomp}
	\end{align}
	Moreover, from Definition \ref{def:bramutoperators}, we see that 
	$$C^{n,\ell} e_n=e_{n-\ell+1},\quad S^{n,\ell}_A e_n=e_{n+\ell},\quad M^{n,\ell}_a e_n=e_{n-\ell}\quad\textrm{and}\quad M^{n,\ell}_A e_n=0.$$ 
	Let us now compute $S^{n,\ell}_a e_n \in \R^{n+\ell+1}$. For $\ell \in [n]$ and $i \in [n+\ell]$, consider the ASG starting with $n$ lines and having a simultaneous $a$-branching involving $\ell$ lines as a first transition. Draw uniformly at random (without replacement) $i$ from the $n+\ell$ lines and assign them type $a$; assign type $A$ to the other $n+\ell-i$ lines. Propagate the types to the $n$ initial lines following the propagation rules in the ASG (see Section \ref{sec:asg}). Then {$(S^{n,\ell}_a e_n)_i$} can be interpreted as the probability of the event where all the $n$ initial lines receive type $a$. The event occurs if and only if, during the type-assignment, 1) the $n-\ell$ lines not resulting from the branching all receive type $a$ and 2) in each of the $\ell$ pairs of lines resulting from the branching, at most one line receives type $A$. In particular, the event can occur only if there are at most $\ell$ lines receiving type $A$, that is, if $n+\ell-i\leq \ell$. We thus get {$(S^{n,\ell}_a e_n)_i=0$} for $i<n$. When $i \in \{n,\cdots,n+\ell\}$, for the event to be realized, there are $\tiny \binom{\ell}{n+\ell-i}$ choices to distribute the type $A$ in the pairs and then $2^{n+\ell-i}$ for the choices of type $A$ lines within those pairs. The total number of choices to distribute the type $A$ among all lines is $\tiny \binom{n+\ell}{n+\ell-i}$. Thus, $${(S^{n,\ell}_a e_n)_i}=2^{n+\ell-i} \binom{\ell}{n+\ell-i}/\binom{n+\ell}{n+\ell-i},\quad \textrm{for $i \in \{n,\cdots,n+\ell\}$}.$$ 
	
	Next, we compute {$D^{n,\ell} e_n \in \R^{n+\ell}$}. For $\ell \in ]\kappa]$ and $i \in [n+\ell-1]_0$, consider the ASG starting with $n$ lines and having a selective branching where a line splits into $\ell$ as a first transition. Draw uniformly at random (without replacement) $i$ from the $n+\ell-1$ lines and assign them type $a$; assign type $A$ to the other $n+\ell-1-i$ lines. Propagate the types to the $n$ initial lines following the propagation rules in the ASG (see Section \ref{sec:asg}). Then {$(D^{n,\ell} e_n)_i$} can be interpreted as the probability of the event where all the $n$ initial lines receive type $a$. The event can occur only if all $n+\ell-1-i$ lines receiving type $A$ are among the $\ell$ lines resulting from the branching and if they do not represent all of those $\ell$ lines (because $p_0^{(\ell)}=0$). We thus get {$(D^{n,\ell} e_n)_i=0$} for $i<n$. When $i \in \{n,\cdots,n+\ell-1\}$, for the event to be realized, {there are} $\tiny \binom{\ell}{n+\ell-1-i}$ choices to distribute the type $A$ into the $\ell$ lines resulting from the branching. In each of these possibilities, the line at the source of the branching receives type $a$ with probability $p_{i+1-n}^{(\ell)}$ and all other lines receive type $a$. The total number of choices to distribute the type $A$ among all lines is $\tiny \binom{n+\ell-1}{n+\ell-1-i}$. Hence, we infer that $${(D^{n,\ell} e_n)_i}=p_{i+1-n}^{(\ell)} \binom{\ell}{n+\ell-1-i}/\binom{n+\ell-1}{n+\ell-1-i},\quad \textrm{for $i \in \{n,\cdots,n+\ell-1\}$}.$$ 
	
	Plugging the obtained expressions of {$C^{n,\ell} e_n$}, $S^{n,\ell}_A e_n$, $M^{n,\ell}_a e_n$, $M^{n,\ell}_A e_n$, $S^{n,\ell}_a e_n$ and $D^{n,\ell} e_n$ into \eqref{1stepdecomp} {and combining with \eqref{momentsloiinvar} and \eqref{eveqmom}} yields
	\begin{align*}
		\alpha_n \rho_n  = &\sum_{\ell=2}^{n} \binom{n}{\ell} \lambda_{n,\ell} \rho_{n-\ell+1} + \sum_{\ell=1}^{n} \binom{n}{\ell} \sigma_{n,\ell}^A \rho_{n+\ell} + \sum_{\ell=1}^{n} \binom{n}{\ell} \left ( m^a_{n,\ell} + \theta_a \mathds{1}_{\ell=1} \right ) \rho_{n-\ell} \\
		& + n \sum_{\ell=2}^{\kappa}\beta_\ell \sum_{k=n}^{n+\ell-1} \left ( \sum_{i=n}^{k} (-1)^{k-i} p_{i+1-n}^{(\ell)} \frac{\binom{n+\ell-1}{i} \binom{n+\ell-1-i}{k-i} \binom{\ell}{n+\ell-1-i}}{\binom{n+\ell-1}{n+\ell-1-i}} \right ) \rho_{k} \nonumber \\
		& + \sum_{\ell=1}^{n} \binom{n}{\ell} \sigma_{n,\ell}^a \sum_{k=n}^{n+\ell} \left ( \sum_{i=n}^{k} (-1)^{k-i} 2^{n+\ell-i} \frac{\binom{n+\ell}{i} \binom{n+\ell-i}{k-i} \binom{\ell}{n+\ell-i}}{\binom{n+\ell}{n+\ell-i}} \right ) \rho_{k}.
	\end{align*}
	{This can be rewritten as} 
	\begin{align*}
		\alpha_n\rho_n =& \sum_{k=1}^{n-1} {\binom{n}{k-1}} \lambda_{n,n-k+1} \rho_{k} + \sum_{k=n+1}^{2n} \binom{n}{k-n} \sigma_{n,k-n}^A \rho_{k} + \sum_{k=0}^{n-1} {\binom{n}{k}} \left ( m^a_{n,n-k} + \theta_a \mathds{1}_{k=n-1} \right ) \rho_{k} \\
		& + \sum_{k=n}^{n+\kappa-1} \left ( \sum_{\ell=2 \vee(k+1-n)}^{\kappa}n\beta_\ell \sum_{i=n}^{k} (-1)^{k-i} p_{i+1-n}^{(\ell)} \frac{\binom{n+\ell-1}{i} \binom{n+\ell-1-i}{k-i} \binom{\ell}{n+\ell-1-i}}{\binom{n+\ell-1}{n+\ell-1-i}} \right ) \rho_{k} \nonumber \\
		& + \sum_{k=n}^{2n} \left ( \sum_{\ell=1\vee(k-n)}^{n} \binom{n}{\ell} \sigma_{n,\ell}^a \sum_{i=n}^{k} (-1)^{k-i} 2^{n+\ell-i} \frac{\binom{n+\ell}{i} \binom{n+\ell-i}{k-i} \binom{\ell}{n+\ell-i}}{\binom{n+\ell}{n+\ell-i}} \right ) \rho_{k}.
	\end{align*} 
	Re-arranging terms we get \eqref{linrelmoments}. 
\end{proof}
\subsection*{Acknowledgment}
We would like to thank Markus Arnoldini for pointing us to literature relevant to the applications of coordinated mutation.
Fernando Cordero was funded by the Deutsche Forschungsgemeinschaft (DFG, German Research Foundation) --- Project-ID 317210226 --- SFB 1283. 
Gr\'egoire V\'echambre was funded by Beijing Natural Science Foundation, project number 1S24067. 

\addtocontents{toc}{\protect\setcounter{tocdepth}{2}}
\bibliographystyle{abbrvnat}
\bibliography{reference2}

\end{document}